\documentclass[10pt,oneside,openright,a4paper,reqno]{amsart}
\usepackage{amsmath}

\usepackage{amsfonts, mathrsfs}
\usepackage{amssymb}

\usepackage[makeroom]{cancel} 
\usepackage{enumerate}
\usepackage{hyperref}
\allowdisplaybreaks

\newtheorem{assumption}{A -}
\newtheorem{anytheorem}{Theorem}[section]
\newtheorem{theorem}[anytheorem]{Theorem}
\newtheorem{lemma}[anytheorem]{Lemma}
\newtheorem{defn}[anytheorem]{Definition} 

\theoremstyle{definition}
\newtheorem{example}[anytheorem]{Example} 
\newtheorem{remark}[anytheorem]{Remark} 

\numberwithin{equation}{section} 

\title[Stochastic anisotropic $p$-Laplace equation]{On Well-Posedness of Stochastic anisotropic $p$-Laplace Equation driven by L\'evy noise}

\author{Neelima}
\address{
School of Mathematics, University of Edinburgh, United Kingdom
}
\email{N.Neelima@sms.ed.ac.uk}

\date{\today}



\begin{document}
\maketitle

\begin{abstract}
In this article, well-posedness of  stochastic anisotropic $p$-Laplace equation driven by L\'evy noise is shown.
Such an equation in deterministic setting has been considered by Lions \cite{lions69}. 
The results obtained in this article can be applied to solve a large class of semilinear and quasilinear stochastic partial differential equations.
\end{abstract}

\section {Introduction and Main Result} \label{sec:introduction}
We establish the well-posedness of stochastic anisotropic $p$-Laplace equation driven by L\'evy noise defined by the following equation, 
\begin{equation}                                                 \label{eq:anisotropic}
\begin{split}
du_t  = \sum_{i=1}^d &  D_i\big(|D_iu_t|^{p_i-2}D_iu_t\big) \,dt 
+ \sum_{j=1}^d \zeta_j |D_j u_t|^\frac{p_j}{2}\,dW_t^j + \sum_{j=1}^\infty h_j(u_t)dW_t^j\\
& +\int_{\mathcal{D}^c}\gamma_t(u_t,z)\tilde{N}(dt,dz)+\int_{\mathcal{D}}\gamma_t(u_t,z)N(dt,dz)
\,\,\, \text{on}\,\,\, (0,T)\times \mathscr{D}, 
\end{split}     
\end{equation}
where $u_t = 0$ on boundary of domain $\mathscr{D} \subset \mathbb{R}^d$ and $u_0$ is a given initial condition. Here, for $i\in\{1,2,\ldots,d\}$, $D_i$ denotes the distributional derivative along the $i$-th coordinate in $\mathbb{R}^d$. Further, $p_i \geq 2$ are real numbers, $\zeta_j$ are constants and $W^j$ are independent Wiener processes on a right continuous complete filtered probability space $(\Omega,\mathscr{F},(\mathscr{F}_t)_{t\in [0,T]},\mathbb{P})$. Also, $N(dt,dz)$ is a Poisson random measure defined on a $\sigma$-finite measure space $(Z,\mathscr{Z},\nu)$ with intensity $\nu$ and $\tilde{N}(dt,dz):=N(dt,dz)-\nu(dz)dt$ is the compensated Poisson random measure. Note that the Poisson random measure $N(dt,dz)$ is independent of the Weiner processes $W^j$. Further, $\mathcal{D}\in \mathscr{Z}$ is such that $\nu(\mathcal{D})<\infty$ and $\mathcal{D}^c=Z \setminus \mathcal{D}$. The term anisotropic signifies that the parameter $p$ in the $p$-Laplace operator takes different values in different directions, which is evident from the drift term of \eqref{eq:anisotropic} as $p_i$'s can be different. The precise assumptions on the functions $h_j$ and $\gamma$ are given in Theorem~\ref{thm:aniso_p_Laplace}. 

Solvability of anisotropic $p$-Laplace equation in deterministic setting, i.e. 
\begin{equation} \label{eq:aniso_pde}
du_t  = \sum_{i=1}^d  D_i\big(|D_iu_t|^{p_i-2}D_iu_t\big) \,dt \,\,\, \text{on}\,\,\, (0,T)\times \mathscr{D}, \,\,\, u_t = 0 \text{ on } \partial \mathscr{D}
\end{equation}
has been studied in Lions \cite{lions69}. Note that if $p_i=p$ for all $i$, then a solution to \eqref{eq:aniso_pde} can be found in the Banach space defined by
\[
W_0^{1,p}(\mathscr{D}):=\{u| u, D_iu \in L^p(\mathscr{D}) , \,i=1,2,\ldots,d ;\,\, u=0 \text{ on } \partial \mathscr{D} \} .
\]
By solution we mean a function $u\in L^p((0,T);W_0^{1,p}(\mathscr{D}))$ such that for every $t\in[0,T]$ and  $\phi \in W_0^{1,p}(\mathscr{D})$,
\[
\int_\mathscr{D}u_t(x)\phi(x)dx= \int_\mathscr{D}u_0(x)\phi(x)dx-\sum_{i=1}^d \int_0^t \int_\mathscr{D}|D_iu_s(x)|^{p-2}D_iu_s(x)D_i\phi(x)dxds \, .
\]
The proof of existence of a solution to PDE \eqref{eq:aniso_pde}, with $p_i=p$ for all $i$, uses the coercivity of the operator $\sum_{i=1}^d   D_i\big(|D_iu|^{p-2}D_iu\big)$, which means there exists a constant $\theta>0$, known as coefficient of coercivity, such that
 \[
 -\sum_{i=1}^d \int_\mathscr{D}|D_iu(x)|^{p}dx \leq -\theta  |u|_{W_0^{1,p}}^p \,.
 \]
 However, when $p_i$'s are different, we can not mimic the above argument as we can not find a $p$ and a space $X$ such that
 \[
 -\sum_{i=1}^d \int_\mathscr{D}|D_iu(x)|^{p_i}dx \leq -\theta  |u|_X^p 
 \]
 holds. To tackle this problem, Lions \cite{lions69} considered the anisotropic p-Laplace operator $\sum_{i=1}^d   D_i\big(|D_iu|^{p_i-2}D_iu\big)$ as a sum of $d$ operators $D_i\big(|D_iu|^{p_i-2}D_iu\big), \,i=1,2,\ldots,d$, where each operator satisfies the coercivity condition with different $p_i,\, \theta_i$ and the space $X_i$, let's call it anisotropic coercivity condition. Then from the appropriate energy equality and anisotropic coercivity condition we get the required a priori estimates. The usual compactness and monotonicity arguments lead to existence of a unique solution of \eqref{eq:aniso_pde} in the space $\cap_{i=1}^d L^{p_i}((0,T);W_0^{1,p_i}(\mathscr{D}))$.  
Results obtained by Pardoux in \cite{pardoux75} can be applied to solve anisotropic $p$-Laplace equation driven by Wiener process. In this article, the technique used in \cite{lions69} is extended to cover the case of anisotropic $p$-Laplace equation \eqref{eq:anisotropic} driven by L\'evy noise  and a unique solution is obtained in the space   
\begin{equation} \label{eq:aniso_sobolev_space}
W_0^{1,\mathbf{p}}(\mathscr{D}):= \{u|u\in L^2(\mathscr{D}), D_iu \in L^{p_i}(\mathscr{D}), \, i=1,2,\ldots,d; u=0 \text{ on } \partial \mathscr{D}  \}.
\end{equation}
We now describe the result in detail. 

%
Let $\mathbb{R}^d$ be a $d$-dimensional Euclidean space and
$\mathscr{D}\subseteq \mathbb{R}^d$ be an open bounded domain with smooth
boundary. 
For any $p\geq 1, \,\, L^p(\mathscr{D})$ is the Lebesgue space
of equivalence classes of real valued measurable functions $u$ defined
on $\mathscr{D}$ such that the norm 
\[|u|_{L^p}:=\Big(\int_\mathscr{D}|u(x)|^p dx\Big)^\frac{1}{p}
\] 
is finite.
Further for $p_i\geq 2$, consider the spaces 
\[ 
W^{x_i,p_i}(\mathscr{D}):=\{u| u \in L^2(\mathscr{D}), D_iu \in L^{p_i}(\mathscr{D}) \}.
\] 
It is then easy to check that the space $W^{x_i, p_i}(\mathscr{D})$  with the norm
\[
|u|_{i,p_i }:=|u|_{L^{2}} + [u]_{i,p_i}
\]
is a Banach space, where $[u]_{i,p_i}:= |D_iu|_{L^{p_i}}$ is a semi-norm.
Let $C_0^\infty(\mathscr{D})$ be the space of smooth functions with compact support in $\mathscr{D}$ and $W^{x_i, p_i}_0(\mathscr{D})$ be its closure in  $W^{x_i,p_i}(\mathscr{D})$. It can be seen that each $W^{x_i, p_i}_0(\mathscr{D})$ is a separable and reflexive Banach space and $W_0^{1,\mathbf{p}}(\mathscr{D})=\cap_{i=1}^d W^{x_i, p_i}_0(\mathscr{D})$ is embedded continuously and densely in the space $L^2(\mathscr{D})$. 

Let $\mathscr{P}$ be the predictable $\sigma$-algebra on $[0,T]\times \Omega$ and $\mathscr{B}(W_0^{1,\mathbf{p}}(\mathscr{D}))$ be the Borel $\sigma$-algebra on $W_0^{1,\mathbf{p}}(\mathscr{D})$.
Assume that $\gamma:[0,T] \times \Omega \times W_0^{1,\mathbf{p}} (\mathscr{D})\times Z \to L^2(\mathscr{D})$ is a $\mathscr{P}\times \mathscr{B}(W_0^{1,\mathbf{p}})\times \mathscr{Z}$-measurable function. 
Finally, $u_0$ is assumed to be a given $L^2(\mathscr{D})$-valued, $\mathscr{F}_0$-measurable random variable. 

Throughout the article, $C$  is a generic constant that may 
change from line to line. Further, 
for a given constant $p\in [1,\infty)$, $L^p(\Omega;X)$ denotes 
the Bochner--Lebesgue space of equivalence classes of random variables 
$x$ taking values in a Banach space $X$ such that the norm
$$|x|_{L^p(\Omega;X)}:=(\mathbb{E}|x|_X^p)^\frac{1}{p}$$
is finite and $L^p((0,T);X)$ denotes the Bochner-Lebesgue space of equivalence classes of $X$-valued measurable functions such that the norm
$$|x|_{L^p((0,T);X)}:= \Big(\int_0^T\!\!|x_t|_X^p\,dt \Big)^\frac{1}{p}<\infty.
 $$
Again, $L^p((0,T)\times\Omega;X)$ denotes the Bochner--Lebesgue space of equivalence classes of $X$-valued stochastic processes which are progressively measurable and the norm
$$|x|_{L^p((0,T)\times\Omega;X)}:=\Big(\mathbb{E}\int_0^T |x_t|_X^p\,dt \Big)^\frac{1}{p}$$  
is finite. Finally, $D([0,T],X)$ denotes the space of $X$-valued c\`{a}dl\`{a}g functions.

\begin{defn}[Solution]\label{def:sol_aniso}
An adapted, c\`{a}dl\`{a}g, $L^2(\mathscr{D})$-valued process $u$ is called a solution of the 
stochastic anisotropic $p$-Laplace equation~\eqref{eq:anisotropic} if
\begin{enumerate}[i)]
\item $dt\times \mathbb{P}$ almost everywhere $u \in W_0^{1,\mathbf{p}}(\mathscr{D})$ and 
\[
\mathbb{E}\int_0^T \int_\mathscr{D} \Big(|u_t(x)|^2 + \sum_{i=1}^d |D_iu_t(x)|^{p_i} \Big)\, dxdt < \infty\, ,
\]
\item for every $t\in [0,T]$ and $\phi \in W_0^{1,\mathbf{p}}(\mathscr{D})$,
\begin{equation*}
\begin{split}
& \int_\mathscr{D}  u_t(x)\phi(x)dx = \int_\mathscr{D} u_0(x)\phi(x)dx - \sum_{i=1}^d \int_0^t \int_\mathscr{D} \big(|D_iu_s|^{p_i-2}D_iu_s(x)D_i\phi(x)dxds \\
& + \sum_{j=1}^d\int_0^t\int_\mathscr{D} \zeta_j|D_ju_s(x)|^\frac{p_j}{2}\phi(x) dx dW_s^j +\sum_{j=1}^\infty \int_0^t\int_\mathscr{D} h_j(u_s(x))\phi(x)dxdW_s^j \\
&+\int_0^t \int_{\mathcal{D}^c}\int_\mathscr{D} \phi(x)\gamma_s(u_s(x),z)dx\tilde{N}(ds,dz)+\int_0^t\int_{\mathcal{D}}\int_\mathscr{D}\phi(x)\gamma_s(u_s(x),z)dxN(ds,dz)  
\end{split}
\end{equation*}
almost surely.
\end{enumerate}
\end{defn}
We formulate the result regarding well-posedness of stochastic anisotropic $p$-Laplace equation \eqref{eq:anisotropic}.

\begin{theorem}\label{thm:aniso_p_Laplace}
 Assume that there exists constants $p_0 \geq \max \{p_1,p_2, \ldots, p_d\}$,  $\zeta_j^2 \leq \frac{2(p_j-1)}{p_j^2(p_0-1)} \wedge \frac{1}{p_0-1}$ and $K>0$ such that almost surely, the following conditions hold for all $t\in[0,T]$.   
\begin{enumerate}
\item For all $u,v \in W_0^{1,\mathbf{p}}(\mathscr{D})$,
\begin{equation}\label{eq:gamma_1}
\int_{\mathcal{D}^c}\int_\mathscr{D}|\gamma_t(u,z)-\gamma_t(v,z)|^2 \,dx \nu(dz)\leq K\int_{\mathscr{D}}|u-v|^2 \, dx  \,.
\end{equation}
\item For all $u \in W_0^{1,\mathbf{p}}(\mathscr{D})$,
\begin{equation} \label{eq:gamma_2}
\int_{\mathcal{D}^c}\int_\mathscr{D}|\gamma_t(u,z)|^2 \,dx \nu(dz) \leq K \Big(1+\int_\mathscr{D}|u|^2 \, dx\Big) \,.
\end{equation}
\item For all $u \in W_0^{1,\mathbf{p}}(\mathscr{D})$,
\begin{equation} \label{eq:gamma_3}
\int_{\mathcal{D}^c}\Big(\int_\mathscr{D}|\gamma_t(u,z)|^2 \,dx \Big)^\frac{p_0}{2} \nu(dz) \leq K \bigg(1+ \Big(\int_\mathscr{D}|u|^2 \, dx \Big)^\frac{p_0}{2}\bigg) \,.
\end{equation}
\end{enumerate} 
Further, if the initial condition $u_0\in L^{p_0}(\Omega;L^2(\mathscr{D}))$ and $h_j:\mathbb{R}\to \mathbb{R},\,j\in\mathbb{N}$ are Lipschitz continuous functions with Lipschitz constants $M_j$ such that the sequence $(M_j)_{j\in\mathbb{N}}\in \ell^2$, then there exists a unique solution of anisotropic $p$-Laplace equation \eqref{eq:anisotropic} in the sense of Definition~\ref{def:sol_aniso}. Furthermore, if $u$ and $\bar{u}$ are two solutions with initial condition $u_0$ and $\bar{u}_0$ respectively, then
\begin{equation} \label{eq:estimates_aniso}
\mathbb{E} \Big( \sup_{t \in [0,T]} |u_t-\bar{u}_t|_{L^2}^p + \sum_{i=1}^d \int_0^T |D_iu_t-D_i \bar{u}_t|_{L^{p_i}}^{p_i} dt \Big) < C \mathbb{E}|u_0-\bar{u}_0|_{L^2}^{p_0} 
\end{equation}   
with $p=2$ in case $p_0=2$ and with any $p\in[2,p_0)$ in case $p_0>2$.
\end{theorem}

The rest of the article is organized as follows. 
In Section~\ref{sec:assumptions}, we formulate and prove our results in abstract framework by considering a large class of SPDEs of the type \eqref{eq:see} satisfying Assumptions A-\ref{ass:hem} to A-\ref{ass: growth_gamma}.
In Section \ref{sec:app}, we show that \eqref{eq:anisotropic} fits in the framework discussed in Section~\ref{sec:assumptions} and hence present a proof of Theorem~\ref{thm:aniso_p_Laplace}. Finally in Section \ref{sec:example}, we give an example of
stochastic partial differential equation which fit 
into the framework of this article but, to the best of our knowledge, can not be solved by using results available so far.

\section{SPDEs in Abstract Framework : Existence \& Uniqueness} \label{sec:assumptions}

Let $(H,(\cdot ,\cdot),|\cdot|_H)$ be a separable Hilbert space, identified with its dual. For $i=1,2,\ldots,k$, let $(V_i,|\cdot|_{V_i})$ be Banach spaces with duals $(V_i^*,|\cdot|_{V_i^*})$ and $\langle \cdot, \cdot \rangle_i$ be the notation for duality pairing between $V_i$ and $V_i^*$. It is well known that the vector space $V:=V_1 \cap V_2 \cap \ldots \cap V_k$ with the norm $|\cdot|_V:=|\cdot|_{V_1}+ |\cdot|_{V_2} \cdots +|\cdot|_{V_k}$ is a Banach space.
Assume that $V$ is separable, reflexive and is 
embedded continuously and densely in $H$.
Thus we obtain the Gelfand triple
\[
V~ \hookrightarrow ~H~\equiv~ H^*~\hookrightarrow ~V^*
\] 
where $\hookrightarrow$ denotes continuous and dense embedding.

We consider the stochastic evolution equation driven by L\'{e}vy noise of the following form:
\begin{equation}                                 \label{eq:see}
\begin{split}
du_t&=\sum_{i=1}^k A_t^i(u_t)dt+\sum_{j=1}^\infty B_t^j(u_t)dW_t^j \\
&+ \int_{\mathcal{D}^c}\gamma_t(u_t,z)\tilde{N}(dt,dz)+\int_{\mathcal{D}}\gamma_t(u_t,z)N(dt,dz), \quad  t \in [0,T]
\end{split}
\end{equation}
where $\mathcal{D}\in \mathscr{Z}$ is such that $\nu(\mathcal{D})<\infty$.  
Here, $A^i, i=1,2,\ldots,k$ are non-linear  operators  mapping $[0,T] \times \Omega \times V_i$ into $V^*_i$,  $B=(B^j)_{j\in \mathbb{N}}$ is a non-linear  operator  mapping $[0,T] \times \Omega \times V $ into $\ell^2(H)$ and $\gamma$ is a non-linear  operator  mapping $[0,T] \times \Omega \times V \times Z$ into $H$. 
Assume that for all $v,w\in V_i$, the processes $\big(\langle A^i_t(v),w \rangle \big)_{t\in[0,T]}$ are progressively measurable and for all $v,w\in V$, $\big((w, B^j_t(v))\big)_{t\in[0,T]}$ are progressively measurable. Since the concept of weak measurability and strong measurability of a mapping coincides if the codomain is separable, we obtain that for all $v\in V_i,\,i=1,2,\ldots,k$, $\big(A^i_t(v)\big)_{t\in[0,T]}$ are progressively measurable. Further, for all $v\in V,\,j\in \mathbb{N}$, $\big(B^j_t(v)\big)_{t\in[0,T]}$ are progressively measurable. Finally, $\gamma$ is assumed to be $\mathscr{P}\times \mathscr{B}(V)\times \mathscr{Z}$-measurable function and  
 $u_0$ is assumed to be a given $H$-valued, $\mathscr{F}_0$-measurable random variable.

Further, we assume that there exist constants 
$\alpha_i>1 \,(i=1,2,\ldots,k), \, \beta \geq 0,\,p_0 \geq \beta+2,\, \theta>0,\,K,\,L',\,L''$ 
and a nonnegative $f\in L^\frac{p_0}{2}((0,T)\times \Omega; \mathbb{R})$ 
such that, almost surely, the following conditions hold 
for all $t\in [0,T]$.
\begin{assumption}[Hemicontinuity]\label{ass:hem} For $i=1,2,\ldots,k$ and $y,x,\bar{x} \in V_i$, the map
\[
\varepsilon \mapsto \langle A_t^i(x+\varepsilon\bar{x}), y \rangle_i
\]
is continuous.
\end{assumption}
\begin{assumption}[Local Monotonicity]\label{ass:lmon} 
For all $x,\bar{x} \in V$,
\begin{equation*}
\begin{split}
2  \sum_{i=1}^k  \langle A_t^i(x)-A_t^i(\bar{x})  , x-\bar{x} \rangle_i + &  \sum_{j=1}^\infty |B^j_t(x)-B^j_t(\bar{x})|_H^2 + \int_{\mathcal{D}^c}|\gamma_t(x,z)-\gamma_t(\bar{x},z)|_H^2 \nu(dz) \\
& \leq \Big[L'+L''\Big(1+\sum_{i=1}^k|\bar{x}|^{\alpha_i}_{V_i}\Big)(1+|\bar{x}|^\beta_H)\Big]|x-\bar{x}|^2_H \, .
\end{split}
\end{equation*}
\end{assumption}
\begin{assumption}[$p_0$-Stochastic Coercivity]\label{ass:coer} 
For all $x$ in $V$,
\[
2\sum_{i=1}^k\langle A_t^i(x), x \rangle_i +(p_0-1)\sum_{j=1}^\infty |B^j_t(x)|_H^2 + \theta \sum_{i=1}^k|x|_{V_i}^{\alpha_i} +\int_{\mathcal{D}^c}|\gamma_t(x,z)|_H^2 \nu(dz)\leq f_t+K|x|^2_H.
\]
\end{assumption}
\begin{assumption}[Growth of $A^i$]\label{ass:groA} For $i=1,2,\ldots,k$ and  $x \in V_i$, 
\[
|A_t^i(x)|_{V^*_i}^{\frac{\alpha_i}{\alpha_i-1}}\leq (f_t+K|x|^{\alpha_i}_{V_i})(1+|x|^\beta_H).
\]
\end{assumption}
\begin{assumption}[Integrability of $\gamma$]\label{ass: growth_gamma} 
For all $x$ in $V$,
\[ 
\int_{\mathcal{D}^c}|\gamma_t(x,z)|_H^{p_0} \nu(dz)\leq f_t^\frac{p_0}{2}+K|x|^{p_0}_H .
\]
\end{assumption}

\begin{remark}\label{rem:groB}
From Assumptions A-\ref{ass:coer} and A-\ref{ass:groA}, we obtain
\begin{align}
&\sum_{j=1}^\infty |B^j_t(x)|_H^2 +\int_{\mathcal{D}^c}|\gamma_t(x,z)|_H^2 \nu(dz)\leq C\Big(1+f_t^\frac{p_0}{2}+|x|_H^{p_0} +\sum_{i=1}^k|x|_{V_i}^{\alpha_i} +|x|_H^\beta\sum_{i=1}^k|x|_{V_i}^{\alpha_i}\Big) \notag 
\end{align}
almost surely for all $t\in[0,T]$ and $x\in V$.
Indeed, using Cauchy-Schwartz inequality, 
Young's inequality and Assumption A-\ref{ass:groA}, 
we obtain that almost surely for all $x\in V$ and $t \in [0,T]$,
\begin{align*}
\sum_{i=1}^k|\langle A_t^i(x),x \rangle_i| & \leq \sum_{i=1}^k \Big[\frac{\alpha_i-1}{\alpha_i} |A_t^i(x)|_{V^*_i}^{\frac{\alpha_i}{\alpha_i-1}}+\frac{1}{\alpha_i}|x|^{\alpha_i}_{V_i} \Big]\notag \\
& \leq \sum_{i=1}^k \Big[\frac{\alpha_i-1}{\alpha_i} \big(f_t+K|x|^{\alpha_i}_{V_i}\big)(1+|x|^\beta_H)+\frac{1}{\alpha_i}|x|^{\alpha_i}_{V_i} \Big]  \\
& \leq C\Big(f_t+\sum_{i=1}^k|x|_{V_i}^{\alpha_i}+|x|_H^\beta\sum_{i=1}^k|x|_{V_i}^{\alpha_i}+f_t^\frac{p_0}{2}+(1+|x|_H)^{p_0}\Big). 
\end{align*}
The above inequality along with Assumption A-\ref{ass:coer} gives the result. In case $p_0=2$, i.e. $\beta=0$, using the similar argument as above, we get
\[
\sum_{j=1}^\infty |B^j_t(x)|_H^2 +\int_{\mathcal{D}^c}|\gamma_t(x,z)|_H^2 \nu(dz)\leq C\Big(f_t+|x|_H^2 +\sum_{i=1}^k|x|_{V_i}^{\alpha_i} \Big)  
\]
almost surely for all $t\in[0,T]$ and $x\in V$.
\end{remark}

\begin{remark}
\label{rem:demicont}	
From Assumptions A-\ref{ass:hem}, A-\ref{ass:lmon} and A-\ref{ass:groA}, we obtain that almost surely for all $t\in[0,T]$ and $i=1,2,\ldots,k$, the operators $A_t^i$ are demicontinuous, 
i.e. $v_n\to v \text{ in } V_i$ implies that 
$A_t^i(v_n)\rightharpoonup A_t^i(v)$ in $V^*_i$. 
This follows using similar arguments as in the proof of Lemma 2.1
in \cite{krylov81}.

One consequence of Remark~\ref{rem:demicont} is that, progressive measurability 
of some process $(v_t)_{t\in [0,T]}$ implies 
the progressive measurability of the processes $\big(A_t^i(v_t)\big)_{t\in [0,T]}$ for all $i=1,2,\ldots,k$. 
\end{remark}

If the driving noise in \eqref{eq:see} is a Wiener process, i.e. intensity $\nu \equiv 0$, then Pardoux \cite{pardoux75} has studied such equations when the operators satisfy hemicontinuity condition A-\ref{ass:hem},  monotonicity condition A-\ref{ass:lmon} (with constant  $L''=0$), coercivity condition A-\ref{ass:coer} (with $p_0=2$, i.e. $\beta=0$), growth assumption A-\ref{ass:groA} (with  $\beta=0$) and an additional assumption on operator $B$ appearing in the stochastic integral term. Note that the noise considered in \cite{pardoux75} is a cylindrical $Q$-Wiener process taking values in a separable Hilbert space. One can see, e.g. in Neelima and \v{S}i\v{s}ka \cite[Appendix A]{neel16}, that the stochastic It\^o integral with respect to cylindrical  $Q$-Wiener process taking values in a separable Hilbert space can be expressed in the form of infinite sum of stochastic It\^o integrals with respect to independent one-dimensional Wiener processes as considered in \eqref{eq:see}.  In view of this fact, the additional condition on operator $B$ assumed in \cite{pardoux75} can be equivalently stated as the following.

\noindent For all $h\in H$ and positive real numbers $N$, there exists a constant $M$ such that for almost all $(t,\omega)\in[0,T]\times \Omega$ and $x,y \in V $ satisfying $|x|_V, |y|_V \leq N$, it holds that  
\begin{equation} \label{eq:weak_local_mono}
\sum_{j=1}^\infty |(h, B_t^j(x))-(h, B_t^j(y))| \leq M |x-y|_V \, . 
\end{equation}
For the case $k=1$, Krylov and Rozovskii \cite{krylov81} generalized the results in \cite{pardoux75} by removing the additional assumption \eqref{eq:weak_local_mono} on the operator $B$. These classical results in \cite{krylov81} have been generalised in number of directions. Gy\"ongy~\cite{gyongy82} extended the results in \cite{krylov81} to include SPDEs driven by c\`adl\`ag semi-martingales and thus allows $\nu$ in \eqref{eq:see} to be different from
zero. Liu and R\"{o}ckner~\cite{rockner10} have extended the framework in \cite{krylov81} to SPDEs with locally monotone operators where the operator $A$, which is the operator acting in the bounded variation term, satisfies a less restrictive growth condition. Thus, authors in \cite{rockner10} allow constants $L''$ and $\beta$, appearing in Assumptions A-\ref{ass:lmon} and A-\ref{ass:groA} respectively, to be non-zero. Brze\'{z}niak, Liu and Zhu~\cite{brz14} generalised the results in \cite{rockner10} to include equations driven by L\'evy noise (i.e. $\nu \not\equiv 0$). However, authors in both \cite{rockner10} and \cite{brz14} have placed an assumption on the growth of the operators appearing under stochastic integrals. Indeed, in the set up of this article, assumption made in \cite{rockner10} can be equivalently stated as: for all $(t,\omega)\in [0,T]\times \Omega$ and $x\in V$,
\begin{equation} \label{eq:B_rockner}
\sum_{j=1}^\infty |B^j_t(x)|_H^2 \leq C(f_t+|x|_H^2) 
\end{equation}
for some $f\in L^\frac{p_0}{2}((0,T)\times \Omega; \mathbb{R})$. Further, assumption made in \cite{brz14} can be stated as: for $f\in L^\frac{p_0}{2}((0,T)\times \Omega; \mathbb{R})$, there exists a constant $\xi < \frac{\theta'}{2 \beta}$ such that for all $(t,\omega)\in [0,T]\times \Omega$ and $x\in V$,
\begin{equation}\label{eq:B_brz}
\sum_{j=1}^\infty |B^j_t(x)|_H^2 +\int_{\mathcal{D}^c}|\gamma_t(x,z)|_H^2 \nu(dz)\leq f_t+C|x|_H^2 + \xi|x|_{V}^{\alpha} 
\end{equation}
where $\theta'$ is the coefficient of coercivity appearing in coercivity assumption made in \cite{brz14}.
In view of Remark~\ref{rem:groB}, the conditions \eqref{eq:B_rockner} and \eqref{eq:B_brz} clearly place a restriction on the growth of operators appearing in stochastic integrals. Recently, for the case $\nu \equiv 0$, Neelima and \v{S}i\v{s}ka \cite{neel16} have overcome this problem by identifying the appropriate coercivity 
assumption as stated in $A$~-~\ref{ass:coer} and proved the existence and uniqueness of solutions to~\eqref{eq:see} (in case $k=1$ and $\nu\equiv 0$) without explicitly restricting the growth of the operator $B$ given in \eqref{eq:B_rockner}. This article is a generalization of \cite{brz14} in two senses:
(a) we do not require the explicit growth condition \eqref{eq:B_brz} to establish existence and uniqueness results, (b) the operator acting in the bounded variation term is of the form $A^1+A^2+\cdots+A^k$, where the operators $A^i$ have different analytic and growth properties.
Again, we have generalized the results in \cite{neel16} by including SPDEs driven by L\'evy noise which satisfy condition (b) stated above, i.e. allowing $k>1$ and $\nu \not \equiv 0$.

In all the above mentioned works, the key to prove the results is the use of an appropriate It\^o formula for the square of the $H$-norm. The formula is an analogue of the energy equality for PDEs which is an essential tool in proving existence and uniqueness theorems for PDEs. The It\^o formula helps in obtaining the a priori estimates under the coercivity and growth assumptions. Under  additional assumptions of monotonicity and hemicontinuity, it helps in proving the existence and uniqueness of the solution. Further, it provides a c\`adl\`ag version of the solution process in the space $H$. 
In this article, using the It\^o formula for processes taking values in intersection of finitely many Banach spaces, given recently by Gy\"ongy and \v{S}i\v{s}ka \cite{gyongy17}, we extend the available results in the literature to include the SPDEs of the type \eqref{eq:see} under the above mentioned assumptions.                                                                                                                                                                                                                                                                                                                                                                                                                                                                                                                                                                                                                                                                                                                                                                                                                                                                                                                                                                                                                                                                                                                                                                                                                                                                                                                                                                                                                                                                                                                                                                                                                                                                                                                                                                                                                                                                                                                                                                                                                                                                                                                                                                                                                                                                                                                                                                                                                                                                                                                                                                                                                                                                                                                                                                                                                                                                                                                                                                                                                                                                                                                                                                                                                                                                                                                                                                                                                                                                                                                                                                                                                                                                                                                                                                                                                                                                                                                                                                                                                                                                                                                                                                                                                                                                                                                                                                                                                                                                                                                                                                                                                                                                                                                                                                                                                                                                                                                                                                                                                                                                                                                                                                                                                                                                                                                                                                                                                                                                                                                                                                                                                                                                                                                                                                                                                                                                                                                                                                                                                                                                                                                                                                                                                                                                                                                                                                                                                                                                                                                                                                                                                                                                                                                                                                                                                                                                                                                                                                                                                                                                                                                                                                                                                                                                                                                                                                                                                                                                                                                                                                                                                                                                                                                                                                                                                                                                                                                                                                                                                                                                                                                                                                                                                                                                                                                                                                                                                                                                                                                                                                                                                                                                                                                                                                                                                                                                                                                                                                                                                                                                                                                                                                                                                                                                                                                                                                                                                                                                                                                                                                                                                                                                                                                                                                                                                                                                                                                                                                                                                                                                                                                                                                                                                                                                                                                                                                                                                                                                                                                                                                                                                                                                                                                                                                                                                                                                                                                                                                                                                                                                                                                                                                                                                                                                                                                                                                                                                                                                                                                                                                                                                                                                                                                                                                                                                                                                                                                                                                                                                                                                                                                                                                                                                                                                                                                                                                                                                                                                                                                                                                                                                                                                                                                                                                                                                                                                                                                                                                                                                                                                                                                                                                                                                                                                                                                                                                                                                                                                                                                                                                                                                                                                                                                                                                                                                                                                                                                                                                                                                                                                                                                                                                                                                                                                                                                                                                                                                                                                                                                                                                                                                                                                                                                                                                                                                                                                                                                                                                                                                                                                                                                                                                                                                                                                                                                                                                                                                                                                                                                                                                                                                                                                                                                                                                                                                                                                                                                                                                                                                                                                                                                                                                                                                                                                                                                                                                                                                                                                                                                                                                                                                                                                                                                                                                                                                                                                                                                                                                                                                                                                                                                                                                                                                                                                                                                                                                                                                                                                                                                                                                                                                                                                                                                                                                                                                                                                                                                                                                                                                                                                                                                                                                                                                                                                                                                                                                                                                                                                                                                                                                                                                                                                                                                                                                                                                                                                                                                                                                                                                                                                                                                                                                                                                                                                                                                                                                                                                                                                                                                                                                                                                                                                                                                                                                                                                                                                                                                                                                                                                                                                                                                                                                                                                                                                                                                                                                                                                                                                                                                                                                                                                                                                                                                                                                                                                                                                                                                                                                                                                                                                                                                                                                                                                                                                                                                                                                                                                                                                                                                                                                                                                                                                                                                                                                                                                                                                                                                                                                                                                                                                                                                                                                                                                                                                                                                                                                                                                                                                                                                                                                                                                                                                                                                                                                                                                                                                                                                                                                                                                                                                                                                                                                                                                                                                                                                                                                                                                                                                                                                                                                                                                                                                                                                                                                                                                                                                                                                                                                                                                                                                                                                                                                                                                                                                                                                                                                                                                                                                                                                                                                                                                                                                                                                                                                                                                                                                                                                                                                                                                                                                                                                                                                                                                                                                                                                                                                                                                                                                   

\begin{defn}[Solution]\label{def:sol}
An adapted, c\`{a}dl\`{a}g, $H$-valued process $u$ is called a solution of the 
stochastic evolution equation~\eqref{eq:see} if
\begin{enumerate}[i)]
\item $dt\times \mathbb{P}$ almost everywhere $u \in V$ with
$$\mathbb{E}\int_0^T(|u_t|_{V_i}^{\alpha_i}+|u_t|_H^2)\, dt < \infty\,, \qquad i=1,2,\ldots,k,$$
\item almost surely 
\[
\int_0^T \left(|u_t|_H^{p_0} + |u_t|_{V_i}^{\alpha_i} |u_t|_H^{p_0-2} \right)\, dt < \infty, \qquad i=1,2,\ldots,k \text{ and }
\]
\item for every $t\in [0,T]$ and $\phi \in V$,
\begin{equation*}
\begin{split}
(u_t,\phi) =& (u_0,\phi) + \sum_{i=1}^k\int_0^t \langle A_s(u_s),\phi \rangle ds + \sum_{j=1}^\infty \int_0^t (\phi,B_s^j(u_s))dW_s^j \\
&+\int_0^t\int_{\mathcal{D}^c}(\phi,\gamma_s(u_s,z))\tilde{N}(ds,dz)+\int_0^t\int_{\mathcal{D}}(\phi,\gamma_s(u_s,z))N(ds,dz) 
\end{split}
\end{equation*}
almost surely.
\end{enumerate}
\end{defn}

The existence and uniqueness of solution to \eqref{eq:see} can be obtained from the existence of a  unique solution to the stochastic evolution equation,
\begin{equation}                                 \label{eq:see_1}
u_t=u_0+ \sum_{i=1}^k\int_0^tA_s^i(u_s)ds+\sum_{j=1}^\infty \int_0^tB_s^j(u_s)dW_s^j +\int_0^t\int_{\mathcal{D}^c}\gamma_s(u_s,z)\tilde{N}(ds,dz) 
\end{equation}
 for $t \in [0,T]$, i.e. the case when the last integral in \eqref{eq:see} vanishes. This is done by means of the interlacing procedure ( see e.g. \cite[Section 4.2]{brz14}).
As a consequence, we will now consider the stochastic evolution equation \eqref{eq:see_1} in rest of the article and prove the existence and uniqueness of solution to  \eqref{eq:see_1} in Theorems~\ref{thm:apriori}, \ref{thm:unique} and \ref{thm:main} below. 
Before that we state two lemmas without proof. 
Lemma~\ref{prop:yor} is a simplified version of Proposition 4.7 in Yor~[\cite{yor91}, Chapter IV] and is used to obtain desired a priori estimates.   
The proof of Lemma~\ref{lem:BDG_jump} can be found in \cite{mikulevicius12}.
\begin{lemma} 
\label{prop:yor}
Let $Y$ be a positive, adapted, right continuous process. If
there exists a constant $K>0$ so that
\[
\mathbb{E}Y_\tau\leq K
\]
for any bounded stopping time $\tau$, then for any $r\in (0,1)$,  
\[
\mathbb{E} \sup_{t\geq 0}Y_t^r \leq \frac{2-r}{1-r}K.
\]
\end{lemma}


\begin{lemma}\label{lem:BDG_jump} 
Let $r\geq 2$ and $T > 0$. There exists a constant $K$, depending only on $r$,
such that for every real-valued, $\mathscr{P} \times \mathscr{Z}$-measurable function $\gamma$ satisfying
\[
\int_0^T \int_Z|\gamma_t(z)|^2 \nu(dz)dt < \infty
\]
almost surely, then the following estimate holds,
\begin{equation}\label{eq:bdg_jump}
\begin{split}
\mathbb{E} \sup_{0\leq t \leq T}\Big|\int_0^t \int_Z  \gamma_s(z) \tilde{N}(ds,dz)\Big|^r \leq & K \mathbb{E}\Big(\int_0^T \int_Z|\gamma_t(z)|^2 \nu(dz)dt \Big)^\frac{r}{2}\\
 & +K \mathbb{E}\int_0^T \int_Z|\gamma_t(z)|^r \nu(dz)dt \, .
\end{split}
\end{equation}
It is known that if $1 \leq r \leq 2$, then the second term in \eqref{eq:bdg_jump} can be dropped.
\end{lemma}

We now show the existence and uniqueness of solution to SPDE \eqref{eq:see_1}.
 

 
                                                                                                               \subsection{A priori Estimates} \label{sec:apriori}

We begin by obtaining some a priori estimates of the solution to SPDE \eqref{eq:see_1}.

\begin{theorem}\label{thm:apriori}
If $u$ is a solution of~\eqref{eq:see_1}, Assumptions A-\ref{ass:coer}, A-\ref{ass:groA} and A-\ref{ass: growth_gamma}  hold, then
\begin{equation}                                  \label{eq:as1}
\begin{split}
\sup_{t\in [0,T]}\mathbb{E}|u_t|_H^{p_0}
& + \sum_{i=1}^k \mathbb{E}\int_0^T|u_t|_H^{p_0-2}|u_t|_{V_i}^{\alpha_i} dt
\leq C\mathbb{E}\Big(|u_0|_H^{p_0}+\! \int_0^T\!\! f_s^\frac{p_0}{2}ds\Big) ,\\
& \sum_{i=1}^k\mathbb{E}\int_0^T|u_t|_{V_i}^{\alpha_i} dt \leq C\mathbb{E}\Big(|u_0|_H^2+\! \int_0^T\!\! f_s \,ds\Big)	\, .
\end{split}
\end{equation}
Moreover,
\begin{equation}                                  \label{eq:as2}
\mathbb{E}\sup_{t\in [0,T]}|u_t|_H^{p}\leq C\mathbb{E}\Big(|u_0|_H^{p_0}+ \int_0^T f_s^\frac{p_0}{2}ds\Big),
\end{equation}
with $p=2$ in case $p_0 = 2$ and with any $p\in [2,p_0)$ in case $p_0>2$, 
where $C$ depends only on $p_0,K,T$ and $\theta$.
\end{theorem}

\begin{proof}
Let $u$ be a solution of \eqref{eq:see_1}  in the sense of Definition~
 \ref{def:sol}. 
In order to obtain higher moment a priori estimates for solutions to \eqref{eq:see_1}, we define for each $n\in\mathbb{N}$,
\begin{equation}
\sigma_n:=\inf\{t\in[0,T]:|u_t|_H>n\}\wedge T \label{eq:stoptime}.
\end{equation}
The solution $u$, being an adapted and c\`adl\`ag $H$-valued process, is bounded on every compact interval. Thus $(\sigma_n)_{n\in\mathbb{N}}$  is a sequence of stopping times converging to 
$T, \,\, \mathbb{P}$- a.s. and $\mathbb{P}\{\sigma_n<T\}=0$ as $n\to~\infty$. 
Applying It\^{o}'s formula for the square of the norm to \eqref{eq:see_1}, see \cite[Theorem 2.1]{gyongy17} and replacing $t$ by $t \wedge \sigma_n$, we get almost surely for all $t\in [0,T]$ and $n\in\mathbb{N}$
\begin{equation} \label{eq:itosq} 
\begin{split}
|& u_{t\wedge\sigma_n} |_H^2=|u_0|_H^2+\int_0^{t\wedge\sigma_n}\Big(2 \sum_{i=1}^k \langle A_s^i(u_s),u_s\rangle_i+\sum_{j=1}^\infty|B_s^j(u_s)|_H^2 \Big)ds\\
&  + 2\sum_{j=1}^\infty\int_0^{t\wedge\sigma_n}(u_s, B_s^j(u_s))dW_s^j  + \int_0^{t\wedge\sigma_n} \int_{\mathcal{D}^c} 2(u_s,\gamma_s(u_s,z)) \tilde{N}(ds,dz) \\
& +\int_0^{t\wedge\sigma_n} \int_{\mathcal{D}^c}|\gamma_s(u_s,z)|_H^2
N(ds,dz) \, . 
\end{split}
\end{equation}
Using the fact $\tilde{N}(dt,dz):=N(dt,dz)-\nu(dz)dt$, we get
\begin{equation} \label{eq:itosq_stopped} 
\begin{split}
|u_{t\wedge\sigma_n}|_H^2 =  &|u_0|_H^2+\int_0^{t\wedge\sigma_n}\Big(2 \sum_{i=1}^k \langle A_s^i(u_s),u_s\rangle_i+\sum_{j=1}^\infty|B_s^j(u_s)|_H^2 \\
& + \int_{\mathcal{D}^c}|\gamma_s(u_s,z)|_H^2 \nu(dz)\Big)ds + 2\sum_{j=1}^\infty\int_0^{t\wedge\sigma_n}(u_s, B_s^j(u_s))dW_s^j  \\
&+ \int_0^{t\wedge\sigma_n} \int_{\mathcal{D}^c} \Big(2(u_s,\gamma_s(u_s,z))+|\gamma_s(u_s,z)|_H^2 \Big)\tilde{N}(ds,dz) \\
\end{split}
\end{equation}
almost surely for all $t\in [0,T]$ and $n\in\mathbb{N}$. 
Notice that this is a $1$-dimensional It\^{o} process. Thus, by It\^{o}'s formula, 
\begin{equation*}
\begin{split}
|& u_{t\wedge\sigma_n} |_H^{p_0}= |u_0|_H^{p_0}+\frac{p_0}{2}\int_0^{t\wedge\sigma_n}|u_s|_H^{p_0-2}\Big(2 \sum_{i=1}^k \langle A_s^i(u_s),u_s\rangle _i +\sum_{j=1}^\infty|B_s^j(u_s)|_H^2 \\
& + \int_{\mathcal{D}^c}|\gamma_s(u_s,z)|_H^2 \nu(dz)\Big)\, ds + p_0\int_0^{t\wedge\sigma_n}|u_s|_H^{p_0-2}\sum_{j=1}^\infty(u_s, B_s^j(u_s))dW_s^j \\
&  +\frac{p_0}{2} \int_0^{t\wedge\sigma_n}\int_{\mathcal{D}^c} |u_s|_H^{p_0-2} \Big[2(u_s,\gamma_s(u_s,z))+|\gamma_s(u_s,z)|_H^2  \Big]\tilde{N}(ds,dz)\\
& +\frac{p_0(p_0-2)}{2}\int_0^{t\wedge\sigma_n}|u_s|_H^{p_0-4}\sum_{j=1}^\infty|(u_s, B_s^j(u_s))|^2\, ds\\
& +\int_0^{t\wedge\sigma_n} \int_{\mathcal{D}^c} \Big[\big||u_s|_H^2 +2(u_s,\gamma_s(u_s,z))+|\gamma_s(u_s,z)|_H^2\big|^\frac{p_0}{2}-|u_s|_H^{p_0} \\
& -\frac{p_0}{2}|u_s|_H^{p_0-2}\big[2(u_s,\gamma_s(u_s,z))+|\gamma_s(u_s,z)|_H^2\big]\Big]N(ds,dz) 
\end{split}
\end{equation*}
almost surely for all $t\in [0,T]$ and $n\in\mathbb{N}$. Again, using the fact $\tilde{N}(dt,dz)=N(dt,dz)-\nu(dz)dt $, we get 
\begin{equation}\label{eq:ito_p0}
\begin{split}
| u_{t\wedge\sigma_n} |_H^{p_0}=  |u_0|_H^{p_0} & + I_1+I_2 + p_0 \sum_{j=1}^\infty \int_0^{t\wedge\sigma_n}|u_s|_H^{p_0-2}(u_s, B_s^j(u_s))dW_s^j \\
&  + p_0 \int_0^{t\wedge\sigma_n}\int_{\mathcal{D}^c} |u_s|_H^{p_0-2} (u_s,\gamma_s(u_s,z))   \tilde{N}(ds,dz)\\
\end{split}
\end{equation}
almost surely for all $t\in [0,T]$ and $n\in\mathbb{N}$, where
\begin{equation*}
\begin{split}
I_1 := &\frac{p_0}{2}\int_0^{t\wedge\sigma_n}|u_s|_H^{p_0-2}\Big(2 \sum_{i=1}^k \langle A_s^i(u_s),u_s\rangle _i +\sum_{j=1}^\infty|B_s^j(u_s)|_H^2 \Big)\, ds \\
& +\frac{p_0(p_0-2)}{2}\int_0^{t\wedge\sigma_n}|u_s|_H^{p_0-4}\sum_{j=1}^\infty|(u_s, B_s^j(u_s))|^2\, ds\\
\text{and} \\
I_2 := & \int_0^{t\wedge\sigma_n} \int_{\mathcal{D}^c} \Big[|u_s+\gamma_s(u_s,z)|_H^{p_0} -|u_s|_H^{p_0}  -p_0|u_s|_H^{p_0-2}(u_s,\gamma_s(u_s,z))\Big]N(ds,dz) \, .
\end{split}
\end{equation*}

Using Cauchy-Schwarz inequality, Assumption A-\ref{ass:coer} and Young's inequality, we get
almost surely for all $t\in [0,T]$ and $n\in\mathbb{N}$
\begin{equation}  \label{eq:I_1}
\begin{split}
I_1 & \leq \frac{p_0}{2}\int_0^{t\wedge\sigma_n}|u_s|_H^{p_0-2}\Big(2 \sum_{i=1}^k \langle A_s^i(u_s),u_s\rangle_i +(p_0-1)\sum_{j=1}^\infty|B_s^j(u_s)|_H^2 \Big)ds \\
& \leq \frac{p_0}{2}\int_0^{t\wedge\sigma_n}|u_s|_H^{p_0-2}\Big(f_s+K|u_s|_H^2-\theta\sum_{i=1}^k|u_s|_{V_i}^{\alpha_i}\Big) ds\\
& \leq \int_0^{t\wedge\sigma_n} \Big(f_s^\frac{p_0}{2} + \frac{p_0(K+1)-2}{2} |u_s|_H^{p_0}-\theta \frac{p_0}{2}\sum_{i=1}^k |u_s|_H^{p_0-2} |u_s|_{V_i}^{\alpha_i} \Big) ds \,.
\end{split}
\end{equation}
We now proceed to estimate $I_2$. Notice that due to Taylor's formula on the map $t \mapsto |x+ty|_H^p$, for any $x, y \in H$ and $p \geq 2$, we get
\[
|x + y|_H^p - |x|_H^p  = \int_0^1 \frac{d}{dt}|x+ty|_H^p dt
\]
and therefore,
\begin{equation}\label{eq:taylor}
\begin{split}
\big| |x + y|_H^p - & |x|_H^p - p|x|_H^{p-2}(x,y)\big|  = p\Big|\int_0^1 \big[|x+ty|_H^{p-2}(x+ty,y)- |x|_H^{p-2}(x,y)\big] dt \Big| \\
& \leq C_p \int_0^1\big(|x|_H^{p-2} + |y|_H^{p-2}\big)|y|_H^2 \, t dt \leq C_p(|x|_H^{p-2}|y|_H^2 + |y|_H^p)\, .
\end{split}
\end{equation}
Now, taking $x=u_s$, $y=\gamma_s(u_s,z)$ and $p=p_0$ in \eqref{eq:taylor}, we get
 \begin{equation*}
 \begin{split}
 &  |u_s+\gamma_s(u_s,z)|_H^{p_0} -|u_s|_H^{p_0}  -p_0|u_s|_H^{p_0-2}(u_s,\gamma_s(u_s,z))  \\ 
 &  \leq C\Big(|u_s|_H^{p_0-2}|\gamma_s(u_s,z)|_H^2+|\gamma_s(u_s,z)|_H^{p_0} \Big) 
 \end{split}
 \end{equation*}
and hence using Young's inequality, we get for all $t\in [0,T]$ and $n\in\mathbb{N}$
\begin{equation} \label{eq:I_2}
\begin{split}
I_2 \leq &  C \int_0^{t\wedge\sigma_n} \int_{\mathcal{D}^c} \Big[|u_s|_H^{p_0-2}|\gamma_s(u_s,z)|_H^2 + |\gamma_s(u_s,z)|_H^{p_0}\Big]N(ds,dz) \\
\leq & C \int_0^{t\wedge\sigma_n} \int_{\mathcal{D}^c} \Big[ |u_s|_H^{p_0}+|\gamma_s(u_s,z)|_H^{p_0}\big]N(ds,dz)\, .
\end{split}
\end{equation}
 Using \eqref{eq:I_1} and \eqref{eq:I_2}, we obtain from  \eqref{eq:ito_p0}
 \begin{equation}\label{eq:ito_p_0_Levy}
 \begin{split}
|u_{t\wedge\sigma_n} & |_H^{p_0} +\theta \frac{p_0}{2}\sum_{i=1}^k \int_0^{t\wedge\sigma_n}|u_s|_H^{p_0-2}|u_s
|_{V_i}^{\alpha_i} ds  \\
\leq & |u_0|_H^{p_0}+ \int_0^{t\wedge\sigma_n} f_s^\frac{p_0}{2}ds+p_0 \sum_{j=1}^\infty \int_0^{t\wedge\sigma_n} |u_s|_H^{p_0-2}(u_s,B_s^j(u_s))dW_s^j \\
& + p_0 \int_0^{t\wedge\sigma_n}\int_{\mathcal{D}^c} |u_s|_H^{p_0-2} (u_s,\gamma_s(u_s,z)) \tilde{N}(ds,dz)\\
& + C \int_0^{t\wedge\sigma_n}\int_{\mathcal{D}^c} \big[|u_s|_H^{p_0}+ |\gamma_s(u_s,z)|_H^{p_0} \big]N(ds,dz)
 \end{split}
 \end{equation}
almost surely for all $t\in [0,T]$ and $n\in\mathbb{N}$. We now aim to apply Lemma \ref{prop:yor}. 
To that end 
let $\tau$ be some bounded stopping time.
 Then in view of Remark \ref{rem:groB} and the fact that $u$ is a solution of equation \eqref{eq:see_1}, it follows that for all $t\in [0,T]$ and $n\in\mathbb{N}$
\[
\mathbb{E}\sum_{j=1}^\infty\int_0^{t\wedge\sigma_n}\mathbf{1}_{\{s\leq \tau\}}|u_s|_H^{p_0-2}(u_s,B_s^j(u_s))dW_s^j=0 
\]
and
\[
\mathbb{E}\int_0^{t\wedge\sigma_n}\int_{\mathcal{D}^c}\mathbf{1}_{\{s\leq \tau\}}|u_s|_H^{p_0-2} (u_s,\gamma_s(u_s,z)) \tilde{N}(ds,dz)=0 \, .
\]
Therefore, replacing $t\wedge \sigma_n$ by $t\wedge\sigma_n \wedge \tau$ in \eqref{eq:ito_p_0_Levy}, taking expectation and using Assumption A-\ref{ass: growth_gamma}                          , we obtain for all $t\in [0,T]$ and $n\in\mathbb{N}$
\begin{equation}  \label{eq:truegronwall}
\begin{split}
\mathbb{E}& |u_{t\wedge\sigma_n \wedge \tau}  |_H^{p_0}+\theta\frac{p_0}{2} \sum_{i=1}^k \mathbb{E}\int_0^{t\wedge \sigma_n\wedge\tau}\!\!|u_s|_H^{p_0-2}|u_s|_{V_i}^{\alpha_i} ds  \\
\leq & \mathbb{E}|u_0|_H^{p_0}+ \mathbb{E}\int_0^{T}f_s^\frac{p_0}{2}ds + C\mathbb{E}\int_0^{t\wedge\sigma_n \wedge \tau}\int_{\mathcal{D}^c} \big[|u_s|_H^{p_0}+ |\gamma_s(u_s,z)|_H^{p_0} \big]\nu(dz)ds \\
 \leq & \mathbb{E}|u_0|_H^{p_0}+ C\mathbb{E}\int_0^{T} f_s^\frac{p_0}{2}ds+C\mathbb{E}\int_0^{t}|u_{s\wedge\sigma_n\wedge\tau}|_H^{p_0}ds
\end{split}
\end{equation}
From this Gronwall's lemma yields
\begin{align}
\label{eq:apriori-after-gronwall}
\mathbb{E}|u_{t\wedge\sigma_n\wedge\tau}|_H^{p_0}
\leq C \mathbb{E}\Big(|u_0|_H^{p_0}+ \int_0^T f_s^\frac{p_0}{2}ds\Big)
\end{align}
for all $t\in[0,T]$ and $n\in\mathbb{N}$.
Letting $n\to \infty$ and using Fatou's lemma, we obtain
\begin{align*} \label{eq:beforeYor}
\mathbb{E}|u_{t\wedge\tau}|_H^{p_0}
&\leq C \mathbb{E}\Big(|u_0|_H^{p_0}+ \int_0^T f_s^\frac{p_0}{2}ds\Big)
\end{align*}
for all $t\in[0,T]$. 
Using Lemma~\ref{prop:yor}, with the process 
$\left(|u_{t}|_H^{p_0}\right)_{t\geq 0}$, we get 
\[
\mathbb{E}\sup_{t\in[0,T]}|u_t|_H^{p_0r}\leq \frac{2-r}{1-r}C\mathbb{E}\Big(|u_0|_H^{p_0}+ \int_0^T f_s^\frac{p_0}{2}ds\Big)
\]
for any $r\in (0,1)$, which proves~\eqref{eq:as2} in case $p_0>2$.

In order to prove~\eqref{eq:as1},
the estimate~\eqref{eq:apriori-after-gronwall} is used in the 
right-hand side of~\eqref{eq:truegronwall} with $\tau=T$ and 
with $n\to \infty$. 
We thus obtain, 
\begin{equation}
\mathbb{E}|u_t|_H^{p_0}
+\theta\frac{p_0}{2}\sum_{i=1}^k \mathbb{E}\int_0^t|u_s|_H^{p_0-2}|u_s|_{V_i}^{\alpha_i} ds
\leq C\mathbb{E}\Big(|u_0|_H^{p_0}+ \int_0^T f_s^\frac{p_0}{2}ds\Big) \label{eq:apriori}
\end{equation}
for all $t\in [0,T]$.
If Assumption A-\ref{ass:coer} holds for some $p_0\geq \beta+2$, then it holds for $p_0=2$ as well. 
Thus, from \eqref{eq:itosq} we obtain
\begin{align*} \label{eq:vbound}
\mathbb{E}|u_t|_H^2+\theta \sum_{i=1}^k \mathbb{E}\int_0^t |u_s|_{V_i}^{\alpha_i} ds \leq \mathbb{E}\Big(|u_0|_H^2+\int_0^Tf_sds\Big)+K\mathbb{E}\int_0^t|u_s|_H^2ds
\end{align*}
for all $t\in[0,T]$. Application of Gronwall's lemma yields
\[
\sup_{t\in [0,T]}\mathbb{E}|u_t|_H^2 \leq C \mathbb{E}\Big(|u_0|_H^2+\int_0^Tf_sds\Big)\, ,
\]
which in turn gives
\[
\theta \sum_{i=1}^k \mathbb{E}\int_0^T |u_s|_{V_i}^{\alpha_i} ds \leq C\mathbb{E}\Big(|u_0|_H^{2}+\int_0^T f_s\,ds\Big)
\]
and hence~\eqref{eq:as1} holds.

To complete the proof it remains to show~\eqref{eq:as2} in case $p_0=2$. 
Considering the sequence of stopping times $\sigma_n$ 
defined in \eqref{eq:stoptime} and using Remark~\ref{rem:groB} 
along with Definition \ref{def:sol}, we observe that the 
stochastic integrals appearing in the right-hand side of \eqref{eq:itosq}
are martingales for each $n \in \mathbb{N}$.
Thus using the Burkholder--Davis--Gundy inequality and Cauchy--Schwartz inequality, we obtain for each $n\in\mathbb{N}$
\begin{equation} \label{eq:p2bdg_weiner}
\begin{split}
\mathbb{E}\sup_{t\in[0, T]}&\Big|\sum_{j=1}^\infty \int_0^{t\wedge \sigma_n}(u_s, B_s^j(u_s))dW_s^j\Big| \\
&\leq 4\mathbb{E}\Big( \sum_{j=1}^\infty \int_0^{T\wedge \sigma_n}|(u_s, B_s^j(u_s))|^2 ds\Big)^\frac{1}{2} \\
&\leq 4\mathbb{E}\Big( \sum_{j=1}^\infty \int_0^{T\wedge \sigma_n}|u_s|_H^2 |B_s^j(u_s)|_H^2 ds\Big)^\frac{1}{2} \, .
\end{split}
\end{equation}
Similarly, for each $n\in\mathbb{N}$
\begin{equation} \label{eq:p2bdg_levy}
\begin{split}
\mathbb{E}\sup_{t\in[0, T]}&\Big| \int_0^{t\wedge \sigma_n}\int_{\mathcal{D}^c}(u_s, \gamma_s(u_s))\tilde{N}(ds,dz)\Big| \\
&\leq C\mathbb{E}\Big(\int_0^{T\wedge \sigma_n}\int_{\mathcal{D}^c}|(u_s, \gamma_s(u_s))|^2 \nu(dz)ds\Big)^\frac{1}{2} \\
&\leq C\mathbb{E}\Big(  \int_0^{T\wedge \sigma_n} \int_{\mathcal{D}^c}|u_s|_H^2 |\gamma_s(u_s)|_H^2 \nu(dz)ds\Big)^\frac{1}{2}  \, .
\end{split}
\end{equation}
Thus \eqref{eq:p2bdg_weiner}, \eqref{eq:p2bdg_levy} along with Remark \ref{rem:groB} and Young's inequality give
\begin{equation} \label{eq:p2bdg}
\begin{split}
&\mathbb{E}\sup_{t\in[0, T]}\Big|\sum_{j=1}^\infty \int_0^{t\wedge \sigma_n}(u_s, B_s^j(u_s))dW_s^j\Big|+\mathbb{E}\sup_{t\in[0, T]}\Big| \int_0^{t\wedge \sigma_n}\int_{\mathcal{D}^c}(u_s, \gamma_s(u_s))\tilde{N}(ds,dz)\Big| \\
& \leq C\mathbb{E}\Big(\sup_{t\in[0, T]}|u_{t\wedge\sigma_n}|_H^2  \int_0^{T\wedge \sigma_n}\big(f_s+|u_s|_H^2+ \sum_{i=1}^k |u_s|_{V_i}^{\alpha_i} \big) ds\Big)^\frac{1}{2} \\
&\leq \epsilon \mathbb{E}\sup_{t\in[0, T]}|u_{t\wedge\sigma_n}|_H^2+C\mathbb{E}\int_0^{T\wedge \sigma_n}\big(f_s+|u_s|_H^2+ \sum_{i=1}^k |u_s|_{V_i}^{\alpha_i} \big) ds 
\end{split}
\end{equation}
for each $n\in\mathbb{N}$.
Moreover, taking supremum and  then expectation in \eqref{eq:itosq} and using Assumption A-\ref{ass:coer} along with \eqref{eq:p2bdg}, we obtain  for each $n\in\mathbb{N}$
\begin{align*}
\mathbb{E}\sup_{t\in[0, T]}|u_{t\wedge \sigma_n}|_H^2  \leq & \epsilon \mathbb{E}\sup_{t\in[0, T]}|u_{t\wedge\sigma_n}|_H^2 \\
&+C\Big(\mathbb{E}|u_0|_H^2+ \mathbb{E}\int_0^{T}\!\! f_s\,ds+ \sum_{i=1}^k\mathbb{E}\int_0^{T} |u_s|_{V_i}^{\alpha_i} ds+\sup_{t\in[0, T]}\mathbb{E}|u_{t}|_H^2\Big). 
\end{align*}
Finally, by choosing $\epsilon$ small and using \eqref{eq:as1} for $p_0=2$, we obtain for each $n\in\mathbb{N}$
\begin{align}
\mathbb{E}\sup_{t\in[0, T]}|u_{t\wedge \sigma_n}|_H^2 \leq C\Big(\mathbb{E}|u_0|_H^2+ \mathbb{E}\int_0^{T}\!\! f_s\,ds\Big) \notag
\end{align}
which on allowing $n\to\infty$ and using Fatou's lemma finishes the proof.
\end{proof} 

Note that we can obtain existence and uniqueness results even if Assumption  A-\ref{ass:coer} is replaced by the following assumption.
\begin{assumption}\label{ass:coer_seminorm} 
For all $x$ in $V$,
\[
2\sum_{i=1}^k\langle A_t^i(x), x \rangle_i +(p_0-1)\sum_{j=1}^\infty |B^j_t(x)|_H^2 + \theta \sum_{i=1}^k[x]_{V_i}^{\alpha_i} +\int_{\mathcal{D}^c}|\gamma_t(x,z)|_H^2 \nu(dz)\leq f_t+K|x|^2_H \, ,
\]
where, $\alpha_i<p_0$ for all $i$ and $[\cdot]_{V_i}$ is a seminorm on the space $V_i$ such that
\[
|\cdot|_{V_i} \leq |\cdot|_H + [\cdot]_{V_i} \,. 
\]
\end{assumption}
In next remark we show that we obtain apriori estimates similar to \eqref{eq:as1} even if Assumption A-\ref{ass:coer} is replaced by A-\ref{ass:coer_seminorm} and then rest of the argument for showing existence and uniqueness of solution to \eqref{eq:see_1} will remain the same. 
\begin{remark}\label{rem:seminorm}
If Assumption A-\ref{ass:coer} is replaced by the A-\ref{ass:coer_seminorm},
then replacing $|u_t|_{V_i}^{\alpha_i}$ by $[u_t]_{V_i}^{\alpha_i}$ everywhere in the proof of Theorem~\ref{thm:apriori}, we obtain
\[
\sum_{i=1}^d \mathbb{E}\int_0^T [u_s^m]_{V_i}^{\alpha_i} ds \leq C\mathbb{E}\Big(|u_0^m|_H^{2}+\int_0^T f_s \,ds\Big)
\]
and 
\[
\mathbb{E}\int_0^T|u_s^m|_{L^2}^{\alpha_i}ds \leq T\mathbb{E}\sup_{s\in [0,T]}|u_s^m|_{L^2}^{\alpha_i}\leq C\mathbb{E}\Big(|u_0^m|_H^{p_0}+ \int_0^T f_s^\frac{p_0}{2}ds\Big)\, 
\]
since  $\alpha_i<p_0$ for all $i$. Thus, 
\begin{equation*}
\begin{split}
\sum_{i=1}^d\mathbb{E}\int_0^T|u_s^m|_{V_i}^{\alpha_i}ds & \leq \sum_{i=1}^d C \Big(\mathbb{E}\int_0^T|u_s^m|_{L^2}^{\alpha_i}ds + \mathbb{E}\int_0^T [u_s^m]_{V_i}^{\alpha_i} ds\Big) \\
& \leq  C\mathbb{E}\Big(|u_0^m|_H^{p_0}+ \int_0^T f_s^\frac{p_0}{2}ds + |u_0^m|_H^{2}+ \int_0^T f_sds\Big)
\end{split}
\end{equation*}
giving all the desired a priori estimates for the solution.
\end{remark}

\subsection{Uniqueness of Solution} \label{sec:unique}
Before stating the result about uniqueness of solution to stochastic evolution equation  \eqref{eq:see_1}, we observe the following.

We note that right hand side in the Assumption A-~\ref{ass:lmon} can be replaced by 
\[
\Big[L \Big(1+\sum_{i=1}^k|\bar{x}|^{\alpha_i}_{V_i}\Big)(1+|\bar{x}|^\beta_H)\Big]|x-\bar{x}|^2_H
\]
for some constant $L$. We use this $L$ in the remaining article.
\begin{defn}
\label{def psi}
Let $\Psi$ be defined as the collection of
$V$-valued and $\mathscr{F}_t$-adapted processes $\psi$ satisfying
\[
\int_0^T\!\!\!\rho(\psi_s)ds< \infty\,\,\, \text{ a.s.}\,, 
\]
where 
\begin{equation*} \label{eq rho}
\rho(x):=L\Big(1+ \sum_{i=0}^k|x|_{V_i}^{\alpha_i}\Big)(1+|x|_H^\beta)	
\end{equation*}
for all $x\in V$.
\end{defn}

Note that if $u$ is a solution to~\eqref{eq:see_1} then $u\in \Psi$. 

\begin{remark}\label{rem:welldef}
For any $\psi \in \Psi$ and $v\in L^2(\Omega ,D([0,T]; H))$,
\begin{align}
\mathbb{E}\Big[\int_0^te^{-\int_0^s \rho(\psi_r)dr}&\rho(\psi_s)|v_s|_H^2ds\Big] \leq \mathbb{E}\sup_{s\in[0, t]}|v_s|_H^2\int_0^te^{-\int_0^s \rho(\psi_r)dr} \rho(\psi_s)ds \notag \\
&=\mathbb{E}\sup_{s\in[0, t]}|v_s|_H^2[1-e^{-\int_0^t \rho(\psi_r)dr}] \leq \mathbb{E}\sup_{s\in[0, t]}|v_s|_H^2 <\infty. \notag
\end{align}
\end{remark}
This remark justifies the existence of the bounded variation integrals 
appearing in the proof of uniqueness that follows.

\begin{theorem} \label{thm:unique}
Let Assumptions A-\ref{ass:lmon} to A-\ref{ass: growth_gamma} hold
and $u_0,\bar{u}_0\in L^{p_0}(\Omega; H)$. 
If $u$ and $\bar{u}$ are two solutions of \eqref{eq:see_1} with $u_0=\bar{u}_0 \,\,\, \mathbb{P}$-a.s., then the processes $u$ and $\bar{u}$ are indistinguishable, i.e. 
\begin{equation}
 \mathbb{P}\Big(\sup_{t\in[0,T]}|u_t-\bar{u}_t|_H=0\Big)=1.             \notag
\end{equation}
\end{theorem}

\begin{proof} Consider two solutions $u$ and $\bar{u}$ of \eqref{eq:see_1}. Thus,
\begin{equation} \label{eq:unique_sol}
\begin{split}
u_t - \bar u_t  = & \sum_{i=1}^k\int_0^t \left(A_s^i(u_s)-A_s^i(\bar{u}_s)\right)\,ds
+\sum_{j=1}^\infty \int_0^t \left(B_s^j(u_s)-B_s^j(\bar{u}_s)\right)\,dW_s^j  \\
&+\int_0^t\int_{\mathcal{D}^c}(\gamma_s(u_s,z)-\gamma_s(\bar{u}_s,z))\tilde{N}(ds,dz)
\end{split}
\end{equation}
almost surely for all $ t \in [0,T]$.
Using the product rule and the It\^{o}'s formula from \cite{gyongy17}, we obtain 
\begin{equation} \label{eq:product}
\begin{split}
d & \Big(e^{-\int_0^t \rho(\bar{u}_s)\,ds}|u_t-\bar{u}_t|_H^2 \Big)= e^{-\int_0^t \rho(\bar{u}_s)ds}\big[d|u_t-\bar{u}_t|_H^2-\rho(\bar{u}_t)|u_t-\bar{u}_t|_H^2\,dt\big]\\
= &  e^{-\int_0^t \rho(\bar{u}_s)ds}\bigg[\Big(2 \sum_{i=1}^k \langle A_t^i(u_t)-A_t^i(\bar{u}_t),u_t-\bar{u}_t\rangle_i+\sum_{j=1}^\infty|B_t^j(u_t) - B_t^j(\bar{u}_t)|_H^2\Big)\,dt  \\
& + \sum_{j=1}^\infty 2\big(u_t-\bar{u}_t,B_t^j(u_t)\!-\!B_t^j(\bar{u}_t)\big) dW_t^j+  \int_{\mathcal{D}^c}2 (u_t-\bar{u}_t,\gamma_t(u_t,z)-\gamma_t(\bar{u}_t,z))\tilde{N}(dt,dz) \\
& + \int_{\mathcal{D}^c}|\gamma_t(u_t,z)-\gamma_t(\bar{u}_t,z)|_H^2
N(dt,dz) -\rho(\bar{u}_t)|u_t-\bar{u}_t|_H^2dt\bigg] 
\end{split}    
\end{equation}
almost surely for all $ t \in [0,T]$. 
For each $n \in \mathbb{N}$, consider the sequence of stopping times $\sigma_n$ given by
\begin{equation} \label{eq:stoptime_uniqueness}
\sigma_n:=\inf\{t\in[0,T]:|u_t|_H>n\}\wedge \inf\{t\in[0,T]:|\bar{u}_t|_H>n\}\wedge T\,.
\end{equation}
Replacing $t$ by $t_n := t\wedge \sigma_n$ in \eqref{eq:product} and taking expectation, we obtain that almost surely for all $ t \in [0,T]$ and $n \in \mathbb{N}$
\begin{equation*}
\begin{split}
\mathbb{E}& \Big ( e^{-\int_0^{t_n}  \rho(\bar{u}_s)\,ds}  |u_{t_n}-\bar{u}_{t_n}|_H^2 \Big)- \mathbb{E}|u_0-\bar{u}_0|_H^2 \\
   = & \mathbb{E}\int_0^{t_n} e^{-\int_0^s \rho(\bar{u}_r)dr}\Big(2 \sum_{i=1}^k \langle A_s^i(u_s)-A_s^i(\bar{u}_s),u_s-\bar{u}_s\rangle_i  +\sum_{j=1}^\infty|B_s^j(u_s) - B_s^j(\bar{u}_s)|_H^2 \\
 & + \int_{\mathcal{D}^c}|\gamma_s(u_s,z)-\gamma_s(\bar{u}_s,z)|_H^2
\nu(dz) -\rho(\bar{u}_s)|u_s-\bar{u}_s|_H^2 \Big)\,ds  \leq 0
\end{split}    
\end{equation*}
where last inequality follows from Assumption A-\ref{ass:lmon}.
Thus if $u_0=\bar{u}_0 \,\,\,\mathbb{P}$-a.s., then 
\begin{equation*}
\mathbb{E}[e^{-\int_0^{t_n} \rho(\bar{u}_s)ds}|u_{t_n}-\bar{u}_{t_n}|_H^2]\leq 0.
\end{equation*} 
Letting $n\to \infty$ and using Fatou's lemma we conclude that
for all $t \in [0,T]$, one has $\mathbb{P}(|u_t - \bar u_t|_H^2=0)=1$.
This, together with the fact that  $u-\bar u$ is c\`adl\`ag in $H$, finishes the 
proof.
\end{proof}   

If we replace the local monotonicity Assumption A-\ref{ass:lmon} by the strong monotonicity  Assumption A-\ref{ass:strong_mono} given below, then we obtain the result about the continuous dependence of the solution to \eqref{eq:see_1} on the initial data as stated in Theorem~\ref{thm:well-posedness}.

\begin{assumption}[Strong Monotonicity] \label{ass:strong_mono} 
There exists a constant $\theta' >0$ such that for all $x,\bar{x} \in V$,
\begin{equation*} 
\begin{split}
2 & \sum_{i=1}^k  \langle A^i(x)-A^i(\bar{x}), x-\bar{x} \rangle_i +  (p_0-1)\sum_{j=1}^\infty|B^j(u)-B^j(v)|_{L^2}^2 \\ & +  \int_{\mathcal{D}^c}|\gamma(u,z)-\gamma(v,z)|_{L^2}^2 \nu(dz)
\leq  -\theta ' \sum_{i=1}^k |x-\bar{x}|_{V_i}^{\alpha_i}+ C |u-v|_{L^2}^2 \, .
\end{split}	
\end{equation*}
\end{assumption}

\begin{theorem} \label{thm:well-posedness}
Let Assumptions A-\ref{ass:groA}, A-\ref{ass: growth_gamma}  and A-\ref{ass:strong_mono} hold
and $u_0,\bar{u}_0\in L^{p_0}(\Omega; H)$. 
If $u$ and $\bar{u}$ are two solutions of  \eqref{eq:see_1} with initial condition $u_0$ and $\bar{u}_0$ respectively, then
\begin{equation*} 
\mathbb{E} \Big( \sup_{t \in [0,T]} |u_t-\bar{u}_t|_{H}^p + \sum_{i=1}^k \int_0^T |u_t- \bar{u}_t|_H^{p_0-2}|u_t- \bar{u}_t|_{V_i}^{\alpha_i} dt \Big) < C \mathbb{E}|u_0-\bar{u}_0|_{H}^{p_0} 
\end{equation*}   
for any $p\in[2,p_0), \, p_0>2$ and
\[
\mathbb{E} \Big( \sup_{t \in [0,T]} |u_t-\bar{u}_t|_{H}^2 + \sum_{i=1}^k \int_0^T |u_t- \bar{u}_t|_{V_i}^{\alpha_i} dt \Big) < C \mathbb{E}|u_0-\bar{u}_0|_{H}^2 \, .
\]
\end{theorem}
\begin{proof}
The proof is very similar to the proof of Theorem~\ref{thm:apriori}.
Indeed we apply It\^o formula from \cite{gyongy17}  to \eqref{eq:unique_sol} and repeat the proof of Theorem~\ref{thm:apriori} for the process $u_t-\bar{u}_t$. Here we note that one needs to use the strong monotonicity Assumption A-\ref{ass:strong_mono} in place of Assumption A-\ref{ass:coer} and work with the sequence of stopping times given by \eqref{eq:stoptime_uniqueness}.
\end{proof}

\subsection{Existence of solution} \label{sec:moment}

We  prove the existence of solution to stochastic evolution equation \eqref{eq:see_1} 
by using the Galerkin method.
We consider a Galerkin scheme $(\mathcal{V}_m)_{m\in \mathbb{N}}$ for $V$, i.e. for each $m\in \mathbb{N}$, $\mathcal{V}_m$ is an $m$-dimensional subspace of $V$ such that $\mathcal{V}_m \subset \mathcal{V}_{m+1} \subset V$ and $\cup_{m\in \mathbb{N}} \mathcal{V}_m$ is dense in $V$. Let $\{\phi_l : \, l=1,2,\ldots m\}$ be a basis of $\mathcal{V}_m$.
Assume that for each $m\in\mathbb{N}$, $u_0^m$ is a $\mathcal{V}_m$-valued 
$\mathscr{F}_0$-measurable random variable satisfying 
\begin{equation}
\label{eq:initial_data}
\sup_{m\in\mathbb{N}}\mathbb{E}|u_0^m|_H^{p_0}<\infty \,\,\, \mbox{and} \,\,\, \mathbb{E}|u_0^m-u_0|_H^2\to 0 \,\,\,\text{as}\,\,\, m\to\infty . 
\end{equation}
It is always possible to obtain such an approximating sequence. For example, consider $\{\phi_l\}_{l\in\mathbb{N}}\!\subset \! V$ forming an orthonormal basis in $H$ and for each $m\in \mathbb{N}$, take $u_0^m=~\Pi_mu_0$ where $\Pi_m:H\to \mathcal{V}_m$ are the projection operators.

For each $m\in \mathbb{N}$ and $\phi_l \in \mathcal{V}_m$, $l=1,2,\ldots,m$, consider the stochastic differential equation:
\begin{equation}\label{eq:sde}
\begin{split}
(u_t^m,\phi_l)= & (u_0^m,\phi_l)+\sum_{i=1}^k\int_0^t \langle A_s^i(u_s^m),\phi_l \rangle_i ds \\
& + \sum_{j=1}^m \int_0^t(\phi_l,B_s^j(u_s^m))dW_s^j +\int_0^t\int_{\mathcal{D}^c}(\phi_l,\gamma_s(u_s^m,z))\tilde{N}(ds,dz) 
\end{split} 
\end{equation}
almost surely for all $t\in[0,T]$.
Using the results on solvability of stochastic differential equations 
in finite dimensional space (see, e.g., Theorem 1 in Gy\"ongy and Krylov \cite{gyongy80}), together with 
Assumptions A-\ref{ass:hem} to A-\ref{ass: growth_gamma} and
Remark~\ref{rem:demicont}, there exists a 
unique adapted and c\`adl\`ag (and thus progressively measurable) $\mathcal{V}_m$-valued process $u^m$ satisfying \eqref{eq:sde}.

\begin{lemma}[A priori Estimates for Galerkin Discretization] \label{lem:galbound}
Suppose that~\eqref{eq:initial_data} and  
Assumptions A-\ref{ass:coer}, A-\ref{ass:groA} and A-\ref{ass: growth_gamma} hold.
Then there exists a constant $C$ independent of $m$, such that
\begin{enumerate}[i)]
\item for every $p_0 \geq \beta+2$,
\begin{equation*}
\sup_{t\in [0,T]}\mathbb{E}|u_t^m|_H^{p_0}+\sum_{i=1}^k\mathbb{E}\int_0^T|u_t^m|_{V_i}^{\alpha_i} \, dt+ \sum_{i=1}^k \mathbb{E}\int_0^T |u_t^m|_H^{p_0-2} |u_t^m|_{V_i}^{\alpha_i} \, dt \leq C.
\end{equation*}
\item Further, 
\begin{equation*}
\mathbb{E}\sup_{t\in [0,T]}|u_t^m|_H^{2}\leq C	
\end{equation*}
and 
\[
\mathbb{E}\sup_{t\in [0,T]}|u_t^m|_H^{p}\leq C
\]
for any $p\in[2,p_0)$ in case $p_0>2$. 
\item Moreover, for all $i=1,2,\ldots, k$
\begin{equation*} 
\mathbb{E}\int_0^{T} |A_s^i(u^m_s)|_{V_i^*}^\frac{\alpha_i}{\alpha_i-1} ds \leq C \quad  
\end{equation*}
\item and finally,
\begin{equation*} 
\mathbb{E}\sum_{j=1}^\infty \int_0^{T} |B^j_s(u^m_s)|_H^2 ds + \mathbb{E} \int_0^T \int_{\mathcal{D}^c}|\gamma_s(u_s^m,z)|_H^2 \nu (dz)ds\leq C. 
\end{equation*}
\end{enumerate}
\end{lemma}

\begin{proof}
Proof of $(i)$ and $(ii)$ is almost a repetition of the proof of analogous results in Theorem \ref{thm:apriori}. Indeed, for each $m,n\in\mathbb{N}$, one can define a sequence of stopping times 
\[
\sigma_n^m:=\inf\{t\in[0,T]:|u_t^m|_H>n\} \wedge T
\] 
and repeat the proof of Theorem \ref{thm:apriori} by replacing $u_t$ with $u_t^m$ and $\sigma_n$ with $\sigma_n^m$. 
There are two main points to be noted. 
First, the stochastic integrals appearing on right-hand side of \eqref{eq:itosq}, 
with $u_s$ replaced by $u_s^m$, are martingales for each $m,n\in\mathbb{N}$. 
Indeed, on a finite dimensional space, all norms are equivalent and hence for each $m,n\in\mathbb{N}$,
\[
\mathbb{E}\int_0^{T\wedge\sigma_n^m} |u_s^m|_V^\alpha ds 
\leq C_m \mathbb{E}\int_0^{T\wedge\sigma_n^m}  n^\alpha ds < \infty
\]
with some constant $C_m$.
The second point is that, since
\[
\sup_{m\in \mathbb{N}}\mathbb{E}|u_0^m|^{p_0} <~\infty,
\]
one can take a constant independent of $m$ to obtain $(i)$
and $(ii)$.
The estimates in $(iii)$ and $(iv)$ can be proved as below. Using Assumption A-\ref{ass:groA}, we obtain
\begin{equation*}
\begin{split}
I  :=\sum_{i=1}^k \mathbb{E}\int_0^{T} & |A_s^i(u^m_s)|_{V^*_i}^\frac{\alpha_i}{\alpha_i-1} ds
\leq \sum_{i=1}^k \mathbb{E}\int_0^{T}(f_s+K|u^m_s|^{\alpha_i}_{V_i})(1+|u^m_s|^\beta_H)ds\\
& = k \mathbb{E}\int_0^{T}f_s\, ds+ k \mathbb{E}\int_0^{T}f_s|u^m_s|^\beta_Hds +K \sum_{i=1}^k \mathbb{E}\int_0^{T}|u^m_s|^{\alpha_i}_{V_i} ds  \\
&\quad + K \sum_{i=1}^k \mathbb{E}\int_0^{T} |u^m_s|^\beta_H |u^m_s|^{\alpha_i}_{V_i} \, ds\, . 
\end{split}
\end{equation*}
Further application of Young's inequality yields
\[
f_s+f_s|u^m_s|_H^\beta \leq \frac{4}{p_0}f_s^\frac{p_0}{2} 
+ \frac{p_0-2}{p_0} + \frac{p_0-2}{p_0}|u^m_s|_H^{\beta \frac{p_0}{p_0-2}}.
\]
Moreover, $|u^m_s|_H^\beta \leq (1+|u^m_s|_H)^{p_0-2}$, since $p_0 \geq \beta +2$.
Hence, 
\begin{equation*}
\begin{split}
I & \leq \frac{4k}{p_0}\mathbb{E}\int_0^{T}f_s^\frac{p_0}{2}ds+\frac{p_0-2}{p_0}kT+ \frac{p_0-2}{p_0}k\mathbb{E}\int_0^T|u^m_s|^{\beta \frac{p_0}{p_0-2}}_H ds\notag \\
&\quad +K \sum_{i=1}^k \mathbb{E}\int_0^{T}|u^m_s|^{\alpha_i}_{V_i}ds+ K \sum_{i=1}^k \mathbb{E}\int_0^T|u^m_s|^{\alpha_i}_{V_i}
(1+|u^m_s|_H)^{p_0-2}ds.	
\end{split}	
\end{equation*}
Furthermore, applying H\"{o}lder's inequality and using the fact $p_0\geq \beta+2$,  
\begin{equation}
\label{eq:boundA}
\begin{split}
I & \leq   \frac{4k}{p_0}\mathbb{E}\int_0^{T}f_s^\frac{p_0}{2}ds+\frac{p_0-2}{p_0}kT+\frac{p_0-2}{p_0}kT^\frac{p_0-2-\beta}{p_0-2}\Big(\mathbb{E}\int_0^T|u^m_s|^{p_0}_H ds\Big)^\frac{\beta}{p_0-2}  \\ 
& \,\,\,\, +(2^{p_0-3}+1)K \sum_{i=1}^k \mathbb{E}\int_0^{T}|u^m_s|^{\alpha_i}_{V_i}ds + 2^{p_0-3}K \sum_{i=1}^k \mathbb{E}\int_0^T|u^m_s|^{\alpha_i}_{V_i}|u_s^m|_H^{p_0-2}ds \\
& \leq \frac{4k}{p_0}\mathbb{E}\int_0^{T}f_s^\frac{p_0}{2}ds+\frac{p_0-2}{p_0}2kT+\frac{p_0-2}{p_0}kT \sup_{0\leq s \leq T}\mathbb{E}|u^m_s|^{p_0}_H \\ 
& \,\,\,\,+(2^{p_0-3}+1)\sum_{i=1}^k K\mathbb{E}\int_0^{T}|u^m_s|^{\alpha_i}_{V_i}ds + 2^{p_0-3}K\sum_{i=1}^k \mathbb{E}\int_0^T|u^m_s|^{\alpha_i}_{V_i}|u_s^m|_H^{p_0-2}ds\,. 
\end{split}	
\end{equation}
By using $(i)$ in \eqref{eq:boundA}, we obtain $(iii)$.
Furthermore, by Remark \ref{rem:groB}, we get
\begin{equation*}
\begin{split}
\mathbb{E}  \int_0^{T}\sum_{j=1}^\infty & |B_s^j(u^m_s)|_H^2 ds  + \mathbb{E} \int_0^T \int_{\mathcal{D}^c}|\gamma_s(u_s^m,z)|_H^2 \nu (dz)ds \\
\leq & C\Big[T+\mathbb{E}\int_0^Tf_s^\frac{p_0}{2}ds +\mathbb{E}\int_0^T|u^m_s|_H^{p_0}ds \\
&  + \sum_{i=1}^k\mathbb{E}\int_0^T|u^m_s|_{V_i}^{\alpha_i} ds + \sum_{i=1}^k \mathbb{E}\int_0^T|u^m_s|_{V_i}^{\alpha_i}(1+|u^m_s|_H)^{p_0-2} ds \Big]  \\
\leq & C\Big[T+\mathbb{E}\int_0^Tf_s^\frac{p_0}{2}ds+T\sup_{s\in[0, T]}\mathbb{E}|u^m_s|_H^{p_0}  \\
&   + \sum_{i=1}^k \mathbb{E}\int_0^T|u^m_s|_{V_i}^{\alpha_i} ds +\sum_{i=1}^k\mathbb{E}\int_0^T|u^m_s|_{V_i}^{\alpha_i}|u^m_s|_H^{p_0-2} ds \Big] 
\end{split}
\end{equation*}
and hence by using $(i)$, we get $(iv)$.
\end{proof}


Having obtained the necessary a priori estimates, we will now extract weakly convergent subsequences using the compactness arguement. After that using the local monotonicity condition, we establish the existence of a solution to \eqref{eq:see_1}.

\begin{lemma} \label{lem:weaklimit}
Let Assumptions A-\ref{ass:lmon} to A-\ref{ass: growth_gamma} 
together with~\eqref{eq:initial_data} hold. 
Then there is a subsequence $(m_q)_{q\in \mathbb{N}}$ and
\begin{enumerate}[i)]
\item there exists a process
$u\in \cap_{i=1}^kL^{\alpha_i}((0,T)\times \Omega ; V_i)$
such that 
\[u^{m_q}\rightharpoonup u\,\,\, \text{in}\,\,\, L^{\alpha_i}((0,T)\times \Omega ; V_i) \,\, \,\,\, \forall \,\, i=1,2,\ldots,k ,
\]
\item there exist  processes $a^i \in L^\frac{\alpha_i}{\alpha_i-1}((0,T)\times \Omega ; V^*)$ such that \[
A^i(u^{m_q})\rightharpoonup a^i \,\,\, \text{in}\,\,\, L^\frac{\alpha_i}{\alpha_i-1}((0,T)\times \Omega ; V^*) \, \,\,\, \forall \,\, i=1,2,\ldots,k ,
\]
\item there exists a process $b\in L^2((0,T)\times \Omega ; l_2(H))$ such that 
\[
B (u^{m_q})\rightharpoonup b \,\,\, \text{in}\,\,\, L^2((0,T)\times \Omega ; l_2(H)),
\]
\item there exists $\Gamma \in L^2((0,T)\times \Omega \times Z; H)$ such that 
\[
\gamma(u^{m_q})1_{\mathcal{D}^c} \rightharpoonup \Gamma 1_{\mathcal{D}^c}  \,\,\, \text{in}\,\,\, L^2((0,T)\times \Omega \times Z ; H).
\]
\end{enumerate}
\end{lemma}

\begin{proof}
The Banach spaces $\,\,L^{\alpha_i}((0,T)\times \Omega ; V_i)$, $ \,\,L^\frac{\alpha_i}{\alpha_i-1}((0,T)\times \Omega ; V^*_i)$, $\,\,L^2((0,T)\times \Omega ; l_2(H))$ and $L^2((0,T)\times \Omega \times Z ; H)$ are reflexive. Thus, due to Lemma \ref{lem:galbound}, there exists a subsequence $m_q$ (see, e.g., Theorem 3.18 in \cite{brezis10}) such that
\begin{itemize}
\item[(i)] $u^{m_q}\rightharpoonup u^i$ in $L^{\alpha_i}((0,T)\times \Omega ; V_i)\,\, \,\,\, \forall \,\, i=1,2,\ldots,k, $
\item[(ii)] $A^i(u^{m_q})\rightharpoonup a^i$ in $L^\frac{\alpha_i}{\alpha_i-1}((0,T)\times \Omega ; V^*_i)\, \,\,\, \forall \,\, i=1,2,\ldots,k,$ 
\item[(iii)] $(B^j(u^{m_q}))_{j=1}^{q}\rightharpoonup (b^j)_{j=1}^\infty$ in $L^2((0,T)\times \Omega ; l_2(H))$ ,
\item[(iv)] $\gamma(u^{m_q})1_{\mathcal{D}^c}\rightharpoonup \Gamma 1_{\mathcal{D}^c}$ in $L^2((0,T)\times \Omega \times Z ; H)$\, .
\end{itemize}
Further, for any $\xi \in V$ and for any adapted and bounded real valued process $\eta_t$, we have for $i,j \in \{1,2,\ldots,k\}$
 \[
 \mathbb{E}\int_0^T\eta_t ( u^i_t-u^j_t,\xi ) dt=\mathbb{E}\int_0^T\eta_t ( u^i_t-u_t^{m_q},\xi ) dt+\mathbb{E}\int_0^T\eta_t ( u_t^{m_q}-u^j_t,\xi ) dt 
 \]
with right--hand--side converging to zero as $q\to \infty$. 
Therefore the processes $u^i, \,\,i=1,2,\ldots,k$ are equal $dt\times \mathbb{P}$~ almost everywhere and henceforth are denoted by $u$ in the remaining article.
\end{proof}


\begin{lemma} \label{lem:limiteq}
Let Assumptions A-\ref{ass:lmon} to A-\ref{ass: growth_gamma}
together with~\eqref{eq:initial_data} hold.
Then  
\begin{enumerate}[i)]
\item for $dt\times \mathbb{P}$ almost everywhere,
\begin{equation}
u_t=u_0+ \sum_{i=1}^k \int_0^t a_s^ids + \sum_{j=1}^\infty \int_0^t b^j_sdW_s^j + \int_0^t \int_{\mathcal{D}^c} \Gamma_s(z)\tilde{N}(ds,dz) \notag
\end{equation}
and moreover almost surely $u\in D([0,T];H)$ and for all  $t\in [0,T]$, 
\begin{equation}\label{eq:itoweak}
\begin{split}
|u_t|_H^2 = & |u_0|_H^2 + \int_0^t \Big[ 2 \sum_{i=1}^k \langle a_s^i, u_s \rangle 
+ \sum_{j=1}^\infty |b^j_s|_H^2 \Big]\, ds 
+ 2\sum_{j=1}^\infty \int_0^t (u_s, b^j_s) dW^j_s \\
& + \int_0^t \int_{\mathcal{D}^c}2 (u_s,\Gamma_s(z))\tilde{N}(ds,dz)+\int_0^t \int_{\mathcal{D}^c}|\Gamma_s(z)|_H^2
N(ds,dz) \, .
\end{split}
\end{equation}
\item Finally, $u\in L^2(\Omega;D([0,T];H))$.
\end{enumerate}
\end{lemma}

\begin{proof}
Using It\^{o}'s isometry, it can be shown that the stochastic integral with respect to Wiener process is a bounded linear operator from $L^2((0,T)\times \Omega ; l_2(H))$ to $L^2((0,T)\times \Omega ; H)$ and hence maps a weakly convergent sequence to a weakly convergent sequence.
Thus, we obtain
\begin{equation}
 \sum_{j=1}^q \int_0^\cdot B^j_s(u^{m_q}_s)dW_s^j \rightharpoonup \sum_{j=1}^\infty \int_0^\cdot  b^j_s dW_s^j      \notag
\end{equation}
in $L^2([0,T]\times \Omega ; H)$, i.e. for any $\psi \in L^2((0,T)\times \Omega ; H)$,
\begin{equation}
 \mathbb{E}\int_0^T\!\!\Big(\sum_{j=1}^q \int_0^t B^j_s(u^{m_q}_s)dW_s^j,\psi(t)\Big)dt\rightarrow  \mathbb{E}\int_0^T\!\Big(\sum_{j=1}^\infty \int_0^t \!\!  b^j_s dW_s^j,\psi(t)\Big)dt.   \label{eq:stochastic}
\end{equation} 
By similar argument, for any $\psi \in L^2((0,T)\times \Omega ; H)$ we have
\begin{equation} \label{eq:stochastic_poisson}
\begin{split}
 \mathbb{E}\int_0^T \Big(\int_0^t \int_{\mathcal{D}^c}  \gamma_s(u^{m_q}_s,z) & \tilde{N}(ds,dz),\psi(t)\Big)dt \\
 & \rightarrow  \mathbb{E}\int_0^T\Big( \int_0^t\int_{\mathcal{D}^c} \Gamma_s(z) \tilde{N}(dz,ds),\psi(t)\Big)dt.  
 \end{split}
\end{equation} 
Similarly, using Holder's inequality it can be shown that for each $i=1,2,\ldots,k$, the Bochner integral is a bounded linear operator from $L^{\frac{\alpha_i}{\alpha_i-1}}((0,T)\times\Omega;V^*_i)$
 to $L^{\frac{\alpha_i}{\alpha_i-1}}((0,T)\times~\Omega;V^*_i)$ and is thus continuous with respect to weak topologies. Therefore, for any $\psi\in L^{\alpha_i}((0,T)\times\Omega;V_i)$,
 \begin{equation}
 \mathbb{E}\int_0^T\Big\langle \int_0^t A_s^i(u^{m_q}_s)ds,\psi(t)\Big\rangle dt\rightarrow  \mathbb{E}\int_0^T\Big\langle \int_0^t  a_s^i ds,\psi(t)\Big\rangle dt.     \label{eq:bochner}
\end{equation}
Fix $n\in \mathbb{N}$. Then for any $\phi \in \mathcal{V}_n$ and an adapted real valued process $\eta_t$ bounded by a constant $C$, we have for any $q\geq n$, 
\begin{equation*}
\begin{split}
\mathbb{E}\int_0^T  \eta_t(u^{m_q}_t,  \phi) & \, dt  =\mathbb{E} \int_0^T \eta_t\Big[(u^{m_q}_0,\phi) + \sum_{i=1}^k \int_0^t \langle A_s^i(u^{m_q}_s),\phi \rangle\, ds \\
& + \sum_{j=1}^\infty \int_0^t (\phi,B_s^j(u^{m_q}_s))dW_s^j + \int_0^t \int_{\mathcal{D}^c} (\phi, \gamma_s(u_s^{m_q},z))\tilde{N}(ds,dz)\Big]\,dt  . 
\end{split}
\end{equation*} 
Taking the limit $q\rightarrow \infty$ and using~\eqref{eq:initial_data},  \eqref{eq:stochastic}, \eqref{eq:stochastic_poisson} and \eqref{eq:bochner}, we obtain
\begin{align}
\mathbb{E}\int_0^T \eta_t(u_t,\phi)\,dt  =  \mathbb{E}\int_0^T & \eta_t\Big[(u_0,\phi)+\sum_{i=1}^k\int_0^t \langle a_s^i,\phi \rangle \, ds\notag \\
& + \sum_{j=1}^\infty \int_0^t (\phi,b^j_s)\,dW_s^j+ \int_0^t \int_{\mathcal{D}^c} (\phi, \Gamma_s(z))\tilde{N}(ds,dz) \Big]\,dt\notag
\end{align} 
with any $\phi \in \mathcal{V}_n$ and any adapted and bounded real valued process $\eta_t$.
Since $\cup_{n\in\mathbb{N}}\mathcal{V}_n$ is dense in $V$, we obtain 
\begin{equation}
u_t=u_0+ \sum_{i=1}^k\int_0^t a_s^ids + \sum_{j=1}^\infty \int_0^t b^j_s dW_s^j+ \int_0^t \int_{\mathcal{D}^c}  \Gamma_s(z)\tilde{N}(ds,dz)  \label{eq:hsol}
\end{equation}
$dt\times \mathbb{P}$ almost everywhere. 
Using Theorem 2.1 on It\^o's formula from \cite{gyongy17}, 
there exists an $H$-valued c\`adl\`ag modification of the process $u$, denoted again by $u$,
which is equal to the right hand side of \eqref{eq:hsol} almost surely for all $t\in[0,T]$.
Moreover~\eqref{eq:itoweak} holds almost surely for all $t\in [0,T]$. 
This completes the proof of part (i) of the lemma. 
It remains to prove part (ii) of  the lemma. 
To that end, consider the sequence of stopping times $\sigma_n$ defined in \eqref{eq:stoptime}.
Using Burkholder--Davis--Gundy inequality together with Cauchy--Schwartz's and Young's inequalities, we obtain
\begin{equation}\label{eq:weakbdg}
\begin{split}
\mathbb{E}\sup_{t\in[0, T]}\Big|\sum_{j=1}^\infty \int_0^{t\wedge \sigma_n} & (u_s, b_s^j)dW_s^j\Big|  \leq 4\mathbb{E}\Big( \sum_{j=1}^\infty \int_0^{T\wedge \sigma_n}|(u_s, b_s^j)|_H^2 ds\Big)^\frac{1}{2} \\
& \leq 4\mathbb{E}\Big( \sum_{j=1}^\infty \int_0^{T\wedge \sigma_n}|u_s|_H^2 |b_s^j|_H^2 ds\Big)^\frac{1}{2}  \\
& \leq 4\mathbb{E}\Big(\sup_{t\in[0, T]}|u_{t\wedge\sigma_n}|_H^2 \sum_{j=1}^\infty \int_0^{T\wedge \sigma_n}|b_s^j|_H^2 ds\Big)^\frac{1}{2}  \\
&\leq \epsilon \mathbb{E}\sup_{t\in[0, T]}|u_{t\wedge\sigma_n}|_H^2+C\mathbb{E}\sum_{j=1}^\infty\int_0^{T\wedge \sigma_n}|b_s^j|_H^2 ds. 
\end{split}
\end{equation}
Similarly,
\begin{equation}\label{eq:weakbdg_poisson}
\begin{split}
\mathbb{E}\sup_{t\in[0, T]}\Big| \int_0^{t\wedge \sigma_n} &
 \int_{\mathcal{D}^c}  (u_s, \Gamma_s(z))  \tilde{N}(ds,dz)\Big| \\
 & \leq C\mathbb{E}\Big(  \int_0^{T\wedge \sigma_n} \int_{\mathcal{D}^c} |(u_s, \Gamma_s(z))|_H^2 \nu(dz)ds\Big)^\frac{1}{2}\, \\
&  \leq C\mathbb{E}\Big(  \int_0^{T\wedge \sigma_n} \int_{\mathcal{D}^c} |u_s|_H^2 |\Gamma_s(z)|_H^2 \nu(dz) ds\Big)^\frac{1}{2} \\
& \leq C\mathbb{E}\Big(\sup_{t\in[0, T]}|u_{t\wedge\sigma_n}|_H^2  \int_0^{T\wedge \sigma_n} \int_{\mathcal{D}^c}|\Gamma_s(z)|_H^2 \nu(dz)ds\Big)^\frac{1}{2}  \\
&\leq \epsilon \mathbb{E}\sup_{t\in[0, T]}|u_{t\wedge\sigma_n}|_H^2+C\mathbb{E} \int_0^{T\wedge \sigma_n} \int_{\mathcal{D}^c} |\Gamma_s(z)|_H^2 \nu(dz) ds. 
\end{split}
\end{equation}
Replace $t$ by $t\wedge \sigma_n$ in~\eqref{eq:itoweak} and take supremum and then expectation.  On using  H\"{o}lder's inequality along with \eqref{eq:weakbdg} and \eqref{eq:weakbdg_poisson}, we obtain  
\begin{equation*}
\begin{split}
\mathbb{E} & \sup_{t\in[0, T]} |u_{t\wedge \sigma_n}|_H^2 \leq \mathbb{E}|u_0|_H^2 +2 \sum_{i=1}^k\Big( \mathbb{E}\int_0^{T}|a_s^i|^\frac{\alpha_i}{\alpha_i-1}ds\Big)^\frac{\alpha_i-1}{\alpha_i}\Big(\mathbb{E}\int_0^{T}|u_s|_{V_i}^{\alpha_i} ds\Big)^\frac{1}{\alpha_i} \\
& +\epsilon \mathbb{E}\sup_{t\in[0, T]}|u_{t\wedge\sigma_n}|_H^2 +C \mathbb{E}\sum_{j=1}^\infty \int_0^T|b_s^j|_H^2ds + C\mathbb{E} \int_0^{T\wedge \sigma_n} \int_{\mathcal{D}^c} |\Gamma_s(z)|_H^2 \nu(dz) ds
\end{split}
\end{equation*} 
which on choosing $\epsilon$ small enough gives
\begin{equation*}
\begin{split}
\mathbb{E}\sup_{t\in[0, T]}|u_{t\wedge \sigma_n}|_H^2  \leq  C \Big[ & \mathbb{E}|u_0|_H^2 +   \sum_{i=1}^k \Big( \mathbb{E}\int_0^{T}|a_s^i|^\frac{\alpha_i}{\alpha_i-1}ds\Big)^\frac{\alpha_i-1}{\alpha_i}\Big(\mathbb{E}\int_0^{T}|u_s|_{V_i}^{\alpha_i} ds\Big)^\frac{1}{\alpha_i} \\
&  +\mathbb{E}\sum_{j=1}^\infty \int_0^T|b_s^j|_H^2ds + \mathbb{E} \int_0^{T\wedge \sigma_n} \int_{\mathcal{D}^c} |\Gamma_s(z)|_H^2 \nu(dz) ds \Big].
\end{split}
\end{equation*} 
Finally taking $n \to \infty$ and using Fatou's lemma, we obtain
$$\mathbb{E}\sup_{t\in[0, T]}|u_t|_H^2<\infty$$
which finishes the proof.
\end{proof}

From now onwards, we will denote the processes $v$ and $u$ by $u$ for notational convenience. 
In order to prove that the process $u$ is the solution of equation \eqref{eq:see_1}, it remains to show that $dt\times \mathbb{P}$ almost everywhere $A^i(v)=a^i$ for $i=1,2,\ldots,k$, $B^j(v)=b^j$ for all $j\in \mathbb{N}$ and $dt\times \mathbb{P} \times \nu$ almost everywhere $\gamma(v)1_{\mathcal{D}^c}=\Gamma 1_{\mathcal{D}^c}$. 
Recall that $\Psi$ and $\rho$ were given in Definition~\ref{def psi}.

\begin{theorem}[Existence of solution] \label{thm:main}
If Assumptions A-\ref{ass:hem} to A-\ref{ass: growth_gamma} hold and $u_0\in L^{p_0}(\Omega; H)$, then the stochastic evolution equation~\eqref{eq:see_1} has a unique solution. Hence, using interlacing procedure, \eqref{eq:see} has a unique solution.
\end{theorem}
\begin{proof}
Let $\psi \in \cap_{i=1}^kL^{\alpha_i}((0,T)\times \Omega ; V_i)\cap \Psi \cap L^2(\Omega ; D([0,T]; H))$, where $\Psi$ is defined in Definition \ref{def psi}. Then using the product rule and It\^{o}'s formula, we obtain
\begin{equation} \label{eq:ito1}
\begin{split}
\mathbb{E}\big(& e^{-\int_0^t \rho(\psi_s)ds} |u_t|_H^2\big) -\mathbb{E}(|u_0|_H^2)  = \mathbb{E}\Big[\int_0^te^{-\int_0^s \rho(\psi_r)dr} \Big(2\sum_{i=1}^k\langle a_s^i,u_s\rangle_i \\
&  +\sum_{j=1}^\infty |b^j_s|_H^2 +  \int_{\mathcal{D}^c} |\Gamma_s(z)|_H^2 \nu(dz)-\rho(\psi_s)|u_s|_H^2\Big )ds\Big] 
\end{split}
\end{equation}
and
\begin{equation*} \label{eq:ito}
\begin{split}
\mathbb{E}\big( e^{-\int_0^t  \rho(\psi_s)ds} & |u_t^{m_q}|_H^2\big)-\mathbb{E}(|u_0^{m_q}|_H^2)=  \mathbb{E}\Big[\int_0^te^{-\int_0^s \rho(\psi_r)dr} \Big(2 \sum_{i=1}^k \langle A_s^i(u_s^{m_q}),u_s^{m_q}\rangle_i \\
& +\sum_{j=1}^\infty |B_s^j(u_s^{m_q})|_H^2 +\int_{\mathcal{D}^c} |\gamma_s(u_s^{m_q},z)|_H^2 \nu(dz)-\rho(\psi_s)|u_s^{m_q}|_H^2\Big)ds\Big] 
\end{split}
\end{equation*}
for all $t\in [0,T]$.
Note that in view of Remark \ref{rem:welldef}, all the integrals are well defined in what follows.
Moreover,
\begin{equation*}
\begin{split}
\mathbb{E} & \Big[\int_0^te^{-\int_0^s \rho(\psi_r)dr}\Big(2 \sum_{i=1}^k\langle A_s^i(u_s^{m_q}),u_s^{m_q}\rangle_i +\sum_{j=1}^\infty |B_s^j(u_s^{m_q})|_H^2 \\ 
& +\int_{\mathcal{D}^c} |\gamma_s(u_s^{m_q},z)|_H^2 \nu(dz) -\rho(\psi_s)|u_s^{m_q}|_H^2\Big)ds\Big]  \\
= & \mathbb{E}\Big[\int_0^te^{-\int_0^s \rho(\psi_r)dr}\Big(2 \sum_{i=1}^k \langle A_s^i(u_s^{m_q})-A_s^i(\psi_s),u_s^{m_q}-\psi_s\rangle_i +2 \sum_{i=1}^k \langle A_s^i(\psi_s),u_s^{m_q} \rangle_i  \\
& +2 \sum_{i=1}^k \langle A_s^i(u_s^{m_q})-A_s^i(\psi_s),\psi_s\rangle_i +\sum_{j=1}^\infty \big|B_s^j(u_s^{m_q})-B_s^j(\psi_s)\big|_H^2-\sum_{j=1}^\infty|B_s^j(\psi_s)|_H^2 \\
&  +2\sum_{j=1}^\infty\big(B_s^j(u_s^{m_q}),B_s^j(\psi_s)\big) +\int_{\mathcal{D}^c} |\gamma_s(u_s^{m_q},z)-\gamma_s(\psi_s,z)|_H^2 \nu(dz) \\
&-\int_{\mathcal{D}^c} |\gamma_s(\psi_s,z)|_H^2 \nu(dz)+ 2 \int_{\mathcal{D}^c} (\gamma_s(u_s^{m_q},z),\gamma_s(\psi_s,z)) \nu(dz) \\
& -\rho(\psi_s)\left[|u_s^{m_q}-\psi_s|_H^2-|\psi_s|_H^2 
+2(u_s^{m_q},\psi_s)\right]\!\Big)ds\Big] \,.
\end{split}
\end{equation*}
Now one can apply the local monotonicity Assumption A-\ref{ass:lmon} to see that
\begin{equation*}
\begin{split}
& \mathbb{E}\big(e^{-\int_0^t \rho(\psi_s)ds}|u_t^{m_q}|_H^2\big)-\mathbb{E}(|u_0^{m_q}|_H^2) \notag \\
&\leq \mathbb{E}\Big[\int_0^te^{-\int_0^s \rho(\psi_r)dr}\Big( 2 \sum_{i=1}^k\langle A_s^i(\psi_s),u_s^{m_q}\rangle_i+2 \sum_{i=1}^k\langle A_s^i(u_s^{m_q})-A_s^i(\psi_s),\psi_s\rangle_i  \\
&  -\sum_{j=1}^\infty|B_s^j(\psi_s)|_H^2 +2\sum_{j=1}^\infty\big(B_s^j(u_s^{m_q}),B_s^j(\psi_s)\big)-\int_{\mathcal{D}^c}
|\gamma_s(\psi_s,z)|_H^2 \nu(dz) \\
& + 2 \int_{\mathcal{D}^c} (\gamma_s(u_s^{m_q},z),\gamma_s(\psi_s,z)) \nu(dz)+\rho(\psi_s)\big[|\psi_s|_H^2 
 -2(u_s^{m_q},\psi_s)\big] \Big)ds\Big] \,.
 \end{split}
\end{equation*}
Integrating over $t$ from $0$ to $T$, letting $q \rightarrow \infty$ and using the weak lower semicontinuity of the norm we obtain
\begin{equation}\label{eq:weaklimits}
\begin{split}
\mathbb{E} \Big[  \int_0^T & \big(e^{-\int_0^t \rho(\psi_s)ds}|u_t|_H^2-|u_0|_H^2\big)dt\Big] \\
\leq & \liminf_{k\rightarrow\infty}\mathbb{E}\Big[\int_0^T\!\big(e^{-\int_0^t \rho(\psi_s)ds}|u_t^{m_q}|_H^2-|u_0^{m_q}|_H^2\big)dt\Big]  \\
\leq & \mathbb{E}\Big[\int_0^T\int_0^te^{-\int_0^s \rho(\psi_r)dr}\Big( 2 \sum_{i=1}^k\langle A_s^i(\psi_s),u_s\rangle_i+2 \sum_{i=1}^k\langle a_s^i-A_s^i(\psi_s),\psi_s\rangle_i  \\
& -\sum_{j=1}^\infty|B_s^j(\psi_s)|_H^2  +2\sum_{j=1}^\infty(b^j_s,B_s^j(\psi_s))-\int_{\mathcal{D}^c}
|\gamma_s(\psi_s,z)|_H^2 \nu(dz) \\
& + 2 \int_{\mathcal{D}^c} (\Gamma_s(z),\gamma_s(\psi_s,z)) \nu(dz)+\rho(\psi_s)\left[|\psi_s|_H^2 
 -2(u_s,\psi_s)\right] \Big)dsdt\Big].  
 \end{split}
\end{equation}
Integrating from $0$ to $T$ in~\eqref{eq:ito1} and combining this 
with~\eqref{eq:weaklimits} leads to
\begin{equation} \label{eq:limitB}
\begin{split}
\mathbb{E} \Big[\int_0^T & \int_0^t e^{-\int_0^s \rho(\psi_r)dr}\Big( 2 \sum_{i=1}^k\langle a_s^i-A_s^i(\psi_s),u_s-\psi_s \rangle_i +\sum_{j=1}^\infty|B_s^j(\psi_s)-b^j_s|_H^2  \\
 & + \int_{\mathcal{D}^c} |\gamma_s(\psi_s,z)-\Gamma_s(z)|_H^2 \nu(dz)-\rho(\psi_s)|u_s-\psi_s|_H^2  \Big)dsdt \Big]\leq 0. 
 \end{split}
\end{equation}
Further, using the Definition~\ref{def psi} and Lemma~\ref{lem:weaklimit},
\[
u\in \cap_{i=1}^k L^{\alpha_i}((0,T)\times \Omega ; V_i)\cap \Psi \cap L^2(\Omega ; D([0,T]; H))\,.
\]
Taking $\psi = u$ in \eqref{eq:limitB}, we obtain that 
$B^j(u)=b^j$ for all $j\in \mathbb{N}$ and $\gamma(u)1_{\mathcal{D}^c}=\Gamma1_{\mathcal{D}^c} \,.$ 
Let $\eta\in L^\infty((0,T)\times \Omega;\mathbb{R})$, $\phi\in V$, 
$\epsilon\in (0,1)$ and let $\psi=u-\epsilon \eta \phi$. Then
from~\eqref{eq:limitB} we obtain that,
\begin{equation*} \label{eq:limitA}
\begin{split}
\mathbb{E}\Big[\int_0^T \!\! \int_0^t & e^{-\int_0^s \rho(u_r-\epsilon \eta_r \phi)dr}\Big( 
2\epsilon \sum_{i=1}^k\langle a_s^i-A_s^i(u_s-\epsilon \eta_s \phi),\eta_s \phi \rangle_i \notag \\
&\quad-\epsilon^2 \rho(u_s-\epsilon \eta_s \phi)|\eta_s \phi|_H^2  \Big)dsdt\Big]\leq 0.
\end{split} 
\end{equation*}
Now we divide by $\epsilon$ and let $\epsilon\rightarrow  0$. Then,  
using Lebesgue dominated convergence theorem and 
Assumption A-\ref{ass:hem} we get,
\begin{equation*}
\begin{split}
&\mathbb{E}\Big[\int_0^T \int_0^t e^{-\int_0^s \rho(u_r)dr}
2 \eta_s \sum_{i=1}^k\langle a_s^i-A_s^i(u_s),\phi \rangle_i dsdt\Big]\leq 0. 
\end{split}
\end{equation*}
Since this holds for any $\eta\in L^\infty((0,T)\times \Omega ;\mathbb{R})$ and $\phi\in V$, we get that $A^i(u)= a^i$ for all $i=1,2,\ldots,k$ which concludes the proof.
\end{proof}

\section{Stochastic anisotropic $p$-Laplace equation} \label{sec:app}
In this section, we prove Theorem~\ref{thm:aniso_p_Laplace} by showing that stochastic anisotropic $p$-Laplace equation \eqref{eq:anisotropic}, in its weak form, fits in the abstract framework discussed in previous section and hence possesses a unique solution which depends continuously on the initial data.

\begin{proof}[Proof of Theorem~\ref{thm:aniso_p_Laplace}]
For $i=1,2, \ldots,d$, take $V_i:=W_0^{x_i,p_i}(\mathscr{D})$ defined in Section~\ref{sec:introduction} so that the space $V$  is the space $W_0^{1,\mathbf{p}}(\mathscr{D})$ given by \eqref{eq:aniso_sobolev_space}.  
Again for $i=1,2, \ldots,d$, let $A^i:V_i \to V_i^*$ be given by,
\begin{equation*}
A^i(u):=D_i\big(|D_iu|^{p_i-2}D_iu\big)\,.
\end{equation*}
Further, let $B^j:V \to L^2(\mathscr{D})$ be given by,
\begin{equation*}
 B^j(u):= 
 \begin{cases}
 \zeta_j |D_ju|^\frac{p_j}{2}+h_j(u) & \text{for} \ j=1,2,\ldots,d, \\
 h_j(u) & \text{otherwise}. 
 \end{cases}
  \end{equation*}
We note that for $u,v \in V_i$,
\begin{equation}  \label{eq:A_i}
\langle A_i(u),v \rangle_i =-\int_{\mathscr{D}}|D_iu(x)|^{p_i-2}D_iu(x)D_iv(x) dx
\end{equation}
and thus using H\"older's inequality,
\begin{equation*} 
\big| \langle A_i(u),v \rangle_i\big| \leq |u|_{V_i}^{p_i-1}|v|_{V_i}\,.
\end{equation*}
Thus, for every $u\in V^i$, $A^i(u)$ is a well-defined linear operator on $V_i$ such that 
\[
|A_iu|_{V_i^*}\leq |u|_{V_i}^{p_i-1}
\]
which implies that Assumptions A-\ref{ass:hem} and A-\ref{ass:groA} hold with $\alpha_i=p_i$ and $\beta=0$. 

We now verify the local monotonicity condition. From standard calculations for $p$-Laplace operators we obtain for each $i=1,2,\ldots,d$,
\[
\big \langle  D_i\big(|D_iu|^{p_i-2}D_iu\big)- D_i\big(|D_iv|^{p_i-2}D_iv\big),u-v \big\rangle_i + \big|\zeta_i|D_iu|^\frac{p_i}{2}-\zeta_i|D_iv|^\frac{p_i}{2}\big|_{L^2}^2 \leq 0
\]
provided $\zeta_i^2 \leq \frac{4(p_i-1)}{p_i^2}$. Since the functions $h_j,\, j\in\mathbb{N}$ are given to be Lipschitz continuous with Lipschitz constants  $M_j$ such that $(M_j)_j\in \ell^2$, we have
\[
|h_j(u)-h_j(v)|_{L^2}^2 \leq M_j^2|u-v|^2_{L^2}\, .
\]
 Using \eqref{eq:gamma_1}, we get
\begin{equation*}
\int_{\mathcal{D}^c}|\gamma(u,z)-\gamma(v,z)|_{L^2}^2 \, \nu(dz) \leq K |u-v|_{L^2}^2 \, .
\end{equation*}
Therefore,
\begin{equation*}
\begin{split}
2 \sum_{i=1}^d  \langle A^i(u)-A^i(v),u-v \rangle_i+ \sum_{j=1}^\infty|B^j(u)-B^j(v)|_{L^2}^2 + & \int_{\mathcal{D}^c}|\gamma(u,z)-\gamma(v,z)|_{L^2}^2 \nu(dz)\\
& \leq C |u-v|_{L^2}^2
\end{split}	
\end{equation*}
and hence Assumption A-\ref{ass:lmon} is satisfied. 

We now wish to verify the $p_0$-stochastic coercivity condition A-\ref{ass:coer}. However, in view of Remark~\ref{rem:seminorm}, it is enough to verify Assumption A-\ref{ass:coer_seminorm} instead. Taking $v=u$ in \eqref{eq:A_i}, we get
\begin{equation*}
 \langle  A^i(u),u \rangle_i  =-\int_{\mathscr{D}}|D_iu(x)|^{p_i} dx \, .
\end{equation*}
Further, 
\begin{equation*}
\begin{split}
2(p_0-1)\big|\zeta_i|D_iu|^\frac{p_i}{2}\big|_{L^2}^2
& = 2(p_0-1)\zeta_i^2 \int_\mathscr{D} |D_iu(x)|^{p_i}dx .
\end{split}
\end{equation*}
Also, \eqref{eq:gamma_2} gives
\begin{equation*}
\int_{\mathcal{D}^c}|\gamma(u,z)|_{L^2}^2 \, \nu(dz) \leq K (1+|u|_{L^2}^2) \, .
\end{equation*}
Choose $\zeta_i^2< \frac{1}{(p_0-1)}$, so that $\theta_i:=2-2(p_0-1)\zeta_i^2>0$. Then taking  $\theta$ to be the minimum of $\theta_1, \theta_2,\ldots, \theta_d$ we have, 
\begin{equation*} \label{eq:coer_aniso}
\begin{split}
2\sum_{i=1}^d \langle  A^i(u),u \rangle_i+ (p_0-1)\sum_{i=1}^\infty |B^i(u)|_{L^2}^2+\theta \sum_{i=1}^d [u]_{V_i}^{p_i}+ \int_{\mathcal{D}^c}|\gamma(u,z)|_{L^2}^2 \nu(dz) \\
\leq  C (1+|u|_{L^2}^2) \, 
\end{split}
\end{equation*}
where, $[u]_{V_i}^{p_i}:= \int_\mathscr{D} |D_iu(x)|^{p_i}dx$ and thus Assumption A-\ref{ass:coer_seminorm} is satisfied.
Finally, we need to verify Assumption A-\ref{ass: growth_gamma}. Using \eqref{eq:gamma_3}, we have
\begin{equation*}
\int_{\mathcal{D}^c}|\gamma(u,z)|_{L^2}^{p_0} \, \nu(dz) \leq K (1+|u|_{L^2}^{p_0})
\end{equation*}
as desired. Since $u_0\in L^{p_0}(\Omega;L^2(\mathscr{D}))$, in view of Remark~\ref{rem:seminorm} along with Theorems~\ref{thm:apriori}, \ref{thm:unique} and \ref{thm:main}, stochastic anisotropic $p$-Laplace equation \eqref{eq:anisotropic} has a unique solution.

We now show the continuous dependence of the solution on the initial data by proving  \eqref{eq:estimates_aniso}. For this, we show that operators in \eqref{eq:anisotropic} satisfy the strong monotonicity Assumption A-\ref{ass:strong_mono}.
Using the inequality
\[
(|a|^ra-|b|^rb)(a-b) \geq 2^{-r} |a-b|^{r+2} \quad \forall \,\,r \geq 0 , \, a,b \in \mathbb{R} ,
\]  
we have for each $i=1,2,\ldots,d$,
\[
\big \langle  D_i\big(|D_iu|^{p_i-2}D_iu\big)- D_i\big(|D_iv|^{p_i-2}D_iv\big),u-v \big\rangle_i \leq -2^{-(p_i-2)} |D_iu-D_iv|_{L^{p_i}}^{p_i} \, .
\]
Further as discussed above,
\[
\big \langle  D_i\big(|D_iu|^{p_i-2}D_iu\big)- D_i\big(|D_iv|^{p_i-2}D_iv\big),u-v \big\rangle_i + 2(p_0-1)\big|\zeta_i|D_iu|^\frac{p_i}{2}-\zeta_i|D_iv|^\frac{p_i}{2}\big|_{L^2}^2 \leq 0
\]
provided $\zeta_i^2 \leq \frac{2(p_i-1)}{p_i^2(p_0-1)}$. Thus we have for $u,v \in W_0^{1,\mathbf{p}}(\mathscr{D})$,
\begin{equation} \label{eq:strong_mono}
\begin{split}
2 & \sum_{i=1}^d  \langle A^i(u)-A^i(v), u-v \rangle_i +  (p_0-1)\sum_{j=1}^\infty|B^j(u)-B^j(v)|_{L^2}^2 \\ & +  \int_{\mathcal{D}^c}|\gamma(u,z)-\gamma(v,z)|_{L^2}^2 \nu(dz)
\leq  -\theta ' \sum_{i=1}^d |D_iu-D_iv|_{L^{p_i}}^{p_i}+ C |u-v|_{L^2}^2 
\end{split}	
\end{equation}
for any $\theta'$ satisfying $0<\theta'< 2^{-(p_i-2)}$ for all $i$.
Thus,  \eqref{eq:estimates_aniso} follows from Theorem \ref{thm:well-posedness}. 
This concludes the proof of Theorem~\ref{thm:aniso_p_Laplace} and hence establishes the well-posedness of stochastic anisotropic $p$-Laplace equation \eqref{eq:anisotropic}.
\end{proof}

\section{Example} \label{sec:example}

Finally, in this section, we present an example of stochastic evolution equation 
which fits in the framework of this article and yet does not satisfy the assumptions of  \cite{brz14, krylov81} or \cite{rockner10}.
For that we introduce few more notations.

Let $W^{1,p}(\mathscr{D})$ be the Sobolev space of real valued functions $u$,
defined on $\mathscr{D}$, such that the norm
\[
|u|_{1,p}:=\Big(\int_\mathscr{D} \big(|u(x)|^p + |\nabla u(x)|^p\big)\, dx\Big)^\frac{1}{p}
\]
is finite, where $\nabla:=(D_1,D_2,\ldots,D_d)$ denotes the gradient.

The closure of $C_0^\infty(\mathscr{D})$ in $W^{1,p}(\mathscr{D})$ with respect to the norm $|\cdot|_{1,p}$ is denoted by  
$W_0^{1,p}(\mathscr{D})$. 
Friedrichs' inequality (see, e.g. Theorem 1.32 in \cite{roubicek05}) implies
that the norm
\[
|u|_{W_0^{1,p}}:=\Big(\int_\mathscr{D} |\nabla u(x)|^p\,dx\Big)^\frac{1}{p} 
\]
is equivalent to $|u|_{1,p}$
and this  equivalent norm $|u|_{W_0^{1,p}}$ will be used in what follows. 
Moreover, let $W^{-1,p}(\mathscr{D})$ denote the dual of $W^{1,p}_0(\mathscr{D})$
and let $|\cdot|_{W^{-1,p}}$ be the norm on this dual space.
It is well known that 
\[ 
W_0^{1,p}(\mathscr{D}) \hookrightarrow L^2(\mathscr{D})\equiv (L^2(\mathscr{D}))^\ast \hookrightarrow W^{-1,p}(\mathscr{D}),
\]
where $\hookrightarrow$ denotes continuous and dense embeddings, 
is a Gelfand triple.

\begin{example}[Quasi-linear equation] 
\label{ex1}
Let $p_1, p_2 > 2$. Assume that there are constants $r,s,t \geq 1$ and continuous function $f^0 $ on $\mathbb{R}$ such that
\begin{equation*}
\begin{split}
& \qquad f^0(x)x \leq  K(1+|x|^{\frac{p_1}{2}+1}); \,\,\, |f^0(x)|\leq K(1+|x|^r) \\
& \text{ and }\,\, 
(f^0(x)-f^0(y))(x-y) \leq K(1+|y|^s)|x-y|^t\,\,\, \forall \, \, x,y\in \mathbb{R}\,.
\end{split}
\end{equation*} 
Let  $h_j:\mathbb{R}\to \mathbb{R},\,j\in\mathbb{N}$ be Lipschitz continuous functions with Lipschitz constants $M_j$ such that the sequence $(M_j)_j\in \ell^2$.
Further, let $Z=\mathbb{R}^d$, $\mathcal{D}^c=\{z\in\mathbb{R}^d: |z| \leq 1 \}$ and $\nu$ be a L\'evy measure on $\mathbb{R}^d$. 
Finally assume that 
 $\gamma:[0,T] \times \Omega \times \mathbb{R} \times Z \to Z$ satisfies
\[
|\gamma_t(x,z)-\gamma_t(y,z)|\leq K|x-y||z| \,\,\, \text{and} \,\,\, |\gamma_t(x,z)| \leq K(1+|x|)|z| 
\]
almost surely, for all $t\in[0,T],\,\, x,y \in \mathbb{R},\,\, z\in \mathcal{D}^c$.

Consider the stochastic partial differential equation,
\begin{equation}                                                 \label{ex:1}
\begin{split}
 du_t  =  & \Big(  \sum_{\ell=1}^d   D_\ell\big(|D_\ell u_t|^{p_1-2}D_\ell u_t\big) - |u_t|^{p_2-2}u_t + f^0(u_t) \Big)\,dt + \sum_{j=1}^d  \zeta |D_j u_t|^\frac{p_1}{2} dW_t^j 
\\
& + \sum_{j=1}^\infty  h_j (u_t) dW_t^j +\int_{\mathcal{D}^c}\gamma_t(u_t,z)\tilde{N}(dt,dz)+\int_{\mathcal{D}}\gamma_t(u_t,z)N(dt,dz)
\end{split}     
\end{equation}
on $(0,T)\times \mathscr{D}$, where $u_t = 0$ on $\partial \mathscr{D}$ and 
$u_0$ is a given $\mathcal{F}_0$-measurable random variable. Moreover, $W^j$ are independent Wiener processes.
We will now  show that such an equation, in its weak form, fits the assumptions
of the present article if any of the following holds:
\begin{enumerate}[1.]
\item
 $d<p_1,\, r=p_1+1,\, s\leq p_1, \, t=2$ and $u_0\in L^6(\Omega;L^2(\mathscr{D}))$.
\item
 $d>p_1,\, r=\frac{2 p_1}{d}+p_1-1, \, s\leq \min \Big \{\frac{p_1^2(t-2)}{(d-p_1)(p_1-2)}, \frac{p_1(p_1-t)}{(p_1-2)}\Big\},\, 2<t<p_1$ and $u_0\in L^6(\Omega;L^2(\mathscr{D}))$.
\end{enumerate}
\noindent \textbf{Case 1.}
Take $V_1:=W^{1,p_1}_0(\mathscr{D})$, $V_2:=L^{p_2}(\mathscr{D})$ and $V:=V_1\cap V_2$.
Then $(V_i, |\cdot|_{V_i})$ are reflexive and separable Banach spaces such that
\[ 
V \hookrightarrow L^2(\mathscr{D})\equiv (L^2(\mathscr{D}))^\ast \hookrightarrow V^*.
\]
Let $A^1:V_1 \to V_1^*$ and $A^2:V_2 \to V_2^*$ be given by,
 \[
A^1(u):= \sum_{\ell=1}^d D_\ell \big(|D_\ell u|^{p_1-2}D_\ell u\big) + f^0(u) \text{ and } A^2(u):= - |u|^{p_2-2}u \, .
\]
Moreover, $B^j:V \to L^2(\mathscr{D})$ be given by
\begin{equation*}
 B^j(u):= 
 \begin{cases}
 \zeta |D_ju|^\frac{p_1}{2}+h_j(u) & \text{for} \ j=1,2,\ldots,d, \\
 h_j(u) & \text{otherwise} \,.
 \end{cases}
  \end{equation*}
The next step is to show that these operators satisfy
the Assumptions A-\ref{ass:hem} to A-\ref{ass: growth_gamma}.
We immediately notice that A-\ref{ass:hem} holds since $f^0$ is continuous. 

We now wish to verify the local monotonicity condition. As discussed earlier, for each $\ell=1,2,\ldots,d$
\[
\big \langle  D_\ell \big(|D_ \ell u|^{p_1-2}D_\ell u \big)- D_\ell \big(|D_\ell v|^{p_1-2}D_\ell v \big),u-v \big\rangle_1 + \big|\zeta |D_\ell u|^\frac{p_1}{2}-\zeta|D_\ell v|^\frac{p_1}{2}\big|_{L^2}^2 \leq 0
\]
provided $\zeta^2 \leq \frac{4(p_1-1)}{p_1^2}$.
Since the function $-|x|^{p_2-2}x$ is monotonically decreasing, we get
\[
\langle -|u|^{p_2-2}u + |v|^{p_2-2}v, u-v \rangle_2 \leq 0. 
\]
Further for $d<p_1$, by Sobolev embedding we have $V_1 \subset L^\infty(\mathscr{D}) $  and therefore using the assumptions imposed on $f_0$ taking $t=2$, we observe that
for $u,v\in V$ 
\begin{equation*}
\begin{split}
\langle f^0(u)-f^0(v), & u-v \rangle_1  \leq K \int_{\mathscr{D}} (1+|v(x)|^s)|u(x)-v(x)|^2 dx \\
& \leq K (1+|v|^s_{L^\infty})|u-v|_{L^2}^2  \leq C (1+|v|^{p_1}_{V_1})|u-v|_{L^2}^2 
\end{split}
\end{equation*}
for $s\leq p_1$.
Using Lipschitz continuity of the functions $h_j,\, j\in\mathbb{N}$, we have
\[
|h_j(u)-h_j(v)|_{L^2}^2 \leq M_j^2|u-v|^2_{L^2}\, ,
\]
where $M_j$ are the Lipschitz constants such that $(M_j)_j\in \ell^2$.
Again using assumptions imposed on $\gamma$  and the fact that $\nu$ is a L\'evy measure, we have
\begin{equation*}
\begin{split}
\int_{\mathcal{D}^c}|\gamma(u,z)-\gamma(v,z)|_{L^2}^2 & \nu(dz)
\leq \int_{\mathcal{D}^c} \int_{\mathscr{D}} |u(x)-v(x)|^2|z|^2 dx \nu(dz) \\
& = K \int_{\mathcal{D}^c} |z|^2 \nu(dz) \int_{\mathscr{D}} |u(x)-v(x)|^2 dx \leq C |u-v|_{L^2}^2 \, .
\end{split}
\end{equation*}
Therefore, we have for all $u,v\in V$
\begin{equation*}
\begin{split}
 2 \sum_{i=1}^2  \langle A^i(u)-A^i(v), & u-v \rangle_i + \sum_{j=1}^\infty|B^j(u)-B^j(v)|_{L^2}^2 + \int_{\mathcal{D}^c}|\gamma(u,z)-\gamma(v,z)|_{L^2}^2 \nu(dz)\\
&  \leq C \Big(1+|v|^{p_1}_{V_1}\Big)|u-v|_{L^2}^2 \leq C \Big(1+\sum_{i=1}^2|v|^{p_i}_{V_i}\Big)|u-v|_{L^2}^2 .
\end{split}	
\end{equation*}
Hence Assumption A-\ref{ass:lmon} is satisfied with $\alpha_i:=p_i\,\, (i=1,2)$ and $\beta := 0$.
Again,
\begin{equation*}
\begin{split}
2 \sum_{\ell=1}^d \big\langle & D_\ell\big(|D_\ell u|^{p_1-2}D_\ell u \big),u \big\rangle_1    =-2 \sum_{\ell=1}^d  \int_{\mathscr{D}} |D_\ell u(x)|^{p_1} dx =-2 |u|_{V_1}^{p_1}
\end{split}
\end{equation*}
and similarly,
\[
2 \langle -|u|^{p_2-2}u,u \rangle _2 = -2|u|_{V_2}^{p_2}.
\]
Moreover using assumptions on $f^0$ and Sobolev embedding, we get
\begin{equation*}
\begin{split}
2  \langle f^0(u),u \rangle_1 & \leq  K \int_{\mathscr{D}} (1+|u(x)|^{\frac{p_1}{2}+1})dx \leq K(1+|u|_{L^\infty}^\frac{p_1}{2}|u|_{L^2}) \\
& \leq C(1+|u|_{V_1}^\frac{p_1}{2}|u|_{L^2}) \leq \delta |u|_{V_1}^{p_1}+C(1+|u|_{L^2}^2),
\end{split}
\end{equation*}
where last inequality is obtained using Young's inequality with sufficiently small $\delta >0$. Further, for any $p_0>2$
\[
(p_0-1)\sum_{j=1}^d |\zeta |D_j u|^\frac{p_1}{2}|_{L^2}^2=(p_0-1)\zeta^2 \sum_{j=1}^d\int_{\mathscr{D}} |D_j u(x)|^{p_1} dx=(p_0-1)\zeta^2 |u|_{V_1}^{p_1}.
\]
Furthermore, using assumptions on $\gamma$ and the fact that $\nu$ is a L\'evy measure on $\mathbb{R}^d$, we get
\begin{equation*}
\begin{split}
\int_{\mathcal{D}^c}|\gamma(u,z)|_{L^2}^2 \nu(dz)
& \leq K \int_{\mathcal{D}^c} \int_{\mathscr{D}} |1+u(x)|^2|z|^2 dx \nu(dz) \\
& = K \int_{\mathcal{D}^c} |z|^2 \nu(dz) \int_{\mathscr{D}} |1+ u(x)|^2 dx 
\leq C (1+|u|_{L^2}^2) \, .
\end{split}
\end{equation*}
Choose $\zeta^2< \frac{2-\delta}{(p_0-1)}$, so that $\theta:= 2-(p_0-1)\zeta^2-\delta>0$. Then  we have, 
\begin{equation*} \label{eq:coer_aniso}
\begin{split}
2\sum_{i=1}^2 \langle  A^i(u),u \rangle_i+ (p_0-1)\sum_{j=1}^\infty |B^j(u)|_{L^2}^2+\theta \sum_{i=1}^d |u|_{V_i}^{p_i}+ \int_{\mathcal{D}^c}|\gamma(u,z)|_{L^2}^2 \nu(dz) \\
\leq  C (1+|u|_{L^2}^2) \, .
\end{split}
\end{equation*}
Hence Assumption A-\ref{ass:coer} is satisfied with $\alpha_i:=p_i\,\, (i=1,2)$.
Again, using the assumptions on $\gamma$ and H\"olders's inequality, we have
\begin{equation*}
\begin{split}
&\int_{\mathcal{D}^c}|\gamma(u,z)|_{L^2}^{p_0} \nu(dz) = \int_{\mathcal{D}^c} \Big(\int_{\mathscr{D}} |\gamma(u(x),z)|^2dx \Big)^\frac{p_0}{2} \nu(dz) \\
& \leq K \int_{\mathcal{D}^c} \Big(\int_{\mathscr{D}} |1+u(x)|^2|z|^2 dx\Big)^\frac{p_0}{2} \nu(dz) = K \int_{\mathcal{D}^c} |z|^{p_0} \nu(dz)\Big( \int_{\mathscr{D}} |1+ u(x)|^2 dx\Big)^\frac{p_0}{2} \\
& \leq  C \int_{\mathcal{D}^c} |z|^2 \nu(dz)\Big[1+ \Big( \int_{\mathscr{D}} |u(x)|^2 dx\Big)^\frac{p_0}{2}\Big] \leq C (1+|u|_{L^2}^{p_0}) \, 
\end{split}
\end{equation*}
and hence Assumption A-\ref{ass: growth_gamma} is satisfied.
Note that using H\"older's inequality, we get for $u,v \in V_1$
\begin{equation*}
\begin{split}
\int_{\mathscr{D}} & |D_\ell u(x)|^{p_1-1} |D_\ell v(x)|dx  \leq \Big(\int_{\mathscr{D}} |D_\ell u(x)|^{p_1} dx\Big)^\frac{p_1-1}{p_1}\Big(\int_{\mathscr{D}} |D_\ell v(x)|^{p_1} dx\Big)^\frac{1}{p_1} \\
& \leq \Big(\sum_{\ell=1}^d\int_{\mathscr{D}} |D_\ell u(x)|^{p_1} dx\Big)^\frac{p_1-1}{p_1}\Big(\sum_{\ell=1}^d\int_{\mathscr{D}} |D_\ell v(x)|^{p_1} dx\Big)^\frac{1}{p_1} = |u|_{V_1}^{p_1-1} |v|_{V_1} \, .
\end{split}
\end{equation*}
Further using assumption on $f^0$ taking $r=p_1+1$, H\"older's inequality, Gagliardo--Nirenberg inequality and Sobolev embedding,
\begin{equation*}
\begin{split}
& \int_{\mathscr{D}} |f^0(u(x))||v(x)|dx \leq K \int_{\mathscr{D}} \big(1+|u(x)|^{p_1+1}\big)|v(x)|dx \\
& \leq K |v|_{L^2}+ K|v|_{L^\infty}|u|_{L^{p_1+1}}^{p_1+1} \leq K |v|_{V_1}(1+ |u|_{L^\infty}^{p_1-1}|u|_{L^2}^2) \leq K |v|_{V_1}(1+ |u|_{V_1}^{p_1-1}|u|_{L^2}^2)
\end{split}
\end{equation*}
and hence
\[
|A^1(u)|_{V_1^*} \leq K |u|_{V_1}^{p_1-1} + K(1+ |u|_{V_1}^{p_1-1}|u|_{L^2}^2) \leq K(1+ |u|_{V_1}^{p_1-1})(1+|u|_{L^2}^2) \, .
\]
Again, using H\"older's inequality 
\[
|A^2(u)|_{V_2^*} \leq K |u|_{V_2}^{p_2-1}\, ,
\]
which implies that Assumption A-\ref{ass:groA} holds with $\alpha_i:=p_i\,\, (i=1,2)$ and $\beta=\frac{2p_1}{p_1-1}<4$. Thus taking $p_0=6$ and $u_0\in L^6(\Omega;L^2(\mathscr{D}))$,  in view of Theorems~\ref{thm:apriori}, \ref{thm:unique} and \ref{thm:main}, equation~\eqref{ex:1} has a unique solution
and moreover for any $p < 6$ we have,
\[
\mathbb{E} \Big( \sup_{t \in [0,T]} |u_t|_{L^2}^p + \sum_{i=1}^2 \int_0^T |u_t|_{V_i}^{\alpha_i} dt \Big) < C \left(1+\mathbb{E}|u_0|_{L^2}^6\right).
\] 
\textbf{Case 2.} In the case $d>p_1$, one can obtain the result in a similar manner using the Sobolev embedding $W_0^{1,p_1}(\mathscr{D})\subset L^\frac{dp_1}{d-p_1}(\mathscr{D})$ and interpolation inequalities stated in \cite[Example 2.4 (2)]{brz14}.
\end{example}

\subsection*{Acknowledgements} 
The author is grateful to her PhD supervisor, Dr. David \v{S}i\v{s}ka,  for his useful comments and guidance during the preparation of this article.

\end{document}